\documentclass{amsart}

\usepackage{amsmath,amsfonts,amssymb}
\usepackage{enumitem}
\usepackage{tikz}

\usepackage{fullpage}

\newtheorem{theorem}{Theorem}[subsection]
\newtheorem{lemma}[theorem]{Lemma}
\newtheorem{proposition}[theorem]{Proposition}
\newtheorem{corollary}[theorem]{Corollary}
\newtheorem{conjecture}[theorem]{Conjecture}

\theoremstyle{definition}
\newtheorem{definition}[theorem]{Definition}

\theoremstyle{remark}

\numberwithin{equation}{subsection}

\newcommand{\newword}[1]{\textbf{\emph{#1}}}

\newcommand{\KD}{\ensuremath{\mathrm{KD}}}
\newcommand{\SSYT}{\ensuremath{\mathrm{SSYT}}}
\newcommand{\SSKT}{\ensuremath{\mathrm{SSKT}}}
\newcommand{\Yam}{\ensuremath{\mathrm{Yam}}}
\newcommand{\QYKD}{\ensuremath{\mathrm{QYKD}}}

\newcommand{\D}{\ensuremath{\mathbb{D}}}

\newcommand{\Label}{\ensuremath{\mathcal{L}}}
\newcommand{\Le}{\ensuremath{\mathrm{\mathbf{L}}}}

\newcommand{\wt}{\ensuremath\mathrm{\mathbf{wt}}}
\newcommand{\des}{\ensuremath\mathrm{\mathbf{des}}}
\newcommand{\comp}[1]{\ensuremath\mathbf{#1}}

\newcommand{\Sym}{\ensuremath{\mathcal{S}}}
\newcommand{\Red}{\ensuremath{\mathrm{R}}}

\newcommand{\schubert}{\ensuremath{\mathfrak{S}}}
\newcommand{\key}{\ensuremath{\kappa}}
\newcommand{\fund}{\ensuremath{\mathfrak{F}}}
\newcommand{\kohnert}{\ensuremath{\mathfrak{K}}}

\newcommand{\B}{\ensuremath{\mathcal{B}}}
\newcommand{\e}{\ensuremath{\mathrm{e}}}
\newcommand{\f}{\ensuremath{\mathrm{f}}}
\newcommand{\Ke}{\ensuremath{\mathfrak{e}}}
\newcommand{\Kf}{\ensuremath{\mathfrak{f}}}
\newcommand{\Rect}{\ensuremath{\mathfrak{R}}}
\newcommand{\rect}{\ensuremath{\mathrm{rect}}}

\newcommand{\Koh}{\ensuremath{\mathcal{K}}}

\newlength\cellsize 

\savebox2{
\begin{picture}(8,8)
\put(0,0){\line(1,0){8}}
\put(0,0){\line(0,1){8}}
\put(8,0){\line(0,1){8}}
\put(0,8){\line(1,0){8}}
\end{picture}}

\newcommand\boxify[1]{\def\thearg{#1}\def\nothing{}%
\ifx\thearg\nothing\vrule width0pt height\cellsize depth0pt%
  \else\hbox to 0pt{\usebox2\hss}\fi%
  \vbox to \cellsize{\vss\hbox to \cellsize{\hss$_{#1}$\hss}\vss}}

\savebox3{
\begin{picture}(8,8)
\put(4,4){\circle{8}}
\end{picture}}

\newcommand{\circify}[1]{\def\thearg{#1}\def\nothing{}%
\ifx\thearg\nothing\vrule width0pt height\cellsize depth0pt%
  \else\hbox to 0pt{\usebox3\hss}\fi%
  \vbox to \cellsize{\vss\hbox to \cellsize{\hss$_{#1}$\hss}\vss}}

\newcommand\nullify[1]{\def\thearg{#1}\def\nothing{}%
\ifx\thearg\nothing\vrule width0pt height\cellsize depth0pt%
  \else\hbox to 0pt{\hss}\fi%
  \vbox to \cellsize{\vss\hbox to \cellsize{\hss$_{#1}$\hss}\vss}}

\newcommand\tableau[1]{\vtop{\let\\=\cr
\setlength\baselineskip{-8000pt}
\setlength\lineskiplimit{8000pt}
\setlength\lineskip{0pt}
\halign{&\boxify{##}\cr#1\crcr}}}

\newcommand\cirtab[1]{\vline\vtop{\let\\=\cr
\setlength\baselineskip{-8000pt}
\setlength\lineskiplimit{8000pt}
\setlength\lineskip{0pt}
\halign{&\circify{##}\cr#1\crcr}}}

\newcommand\nulltab[1]{\vtop{\let\\=\cr
\setlength\baselineskip{-8000pt}
\setlength\lineskiplimit{8000pt}
\setlength\lineskip{0pt}
\halign{&\nullify{##}\cr#1\crcr}}}

\newlength\bigcellsize 
\savebox4{
\begin{picture}(12,12)
\put(6,6){\circle{12}}
\end{picture}}

\newcommand{\bigcir}[1]{\def\thearg{#1}\def\nothing{}%
\ifx\thearg\nothing\vrule width0pt height\bigcellsize depth0pt%
  \else\hbox to 0pt{\usebox4\hss}\fi%
  \vbox to \bigcellsize{\vss\hbox to \bigcellsize{\hss$\scriptstyle #1$\hss}\vss}}

\newcommand{\cball}[2]{%
  \begin{tikzpicture}
    \filldraw[fill=#1!35,draw=black] circle (4pt) node {$\scriptstyle #2$};
  \end{tikzpicture}
}

\newcommand{\leftball}[2]{\makebox[0pt]{\raisebox{1.5pt}{$\leftarrow$}}\cball{#1}{#2}}

\usepackage{xcolor}
\usepackage[colorinlistoftodos]{todonotes}

\begin{document}


\title{Demazure crystals for Kohnert polynomials}  

\author{Sami Assaf}
\address{Department of Mathematics, University of Southern California, 3620 S. Vermont Ave., Los Angeles, CA 90089-2532, U.S.A.}
\email{shassaf@usc.edu}
\thanks{Work supported in part by NSF DMS-1763336.}


\date{\today}


\keywords{Demazure characters, Kohnert polynomials, Schubert polynomials}

\begin{abstract}
  Kohnert polynomials are polynomials indexed by unit cell diagrams in the first quadrant defined earlier by the author and Searles that give a common generalization of Schubert polynomials and Demazure characters for the general linear group. Demazure crystals are certain truncations of normal crystals whose characters are Demazure characters. For each diagram satisfying a southwest condition, we construct a Demazure crystal whose character is the Kohnert polynomial for the given diagram, resolving an earlier conjecture of the author and Searles that these polynomials expand nonnegatively into Demazure characters. We give explicit formulas for the expansions with applications including a characterization of those diagrams for which the corresponding Kohnert polynomial is a single Demazure character.
\end{abstract}

\maketitle

%
\section{Introduction}
%
\label{sec:introduction}

Given a polynomial expressed as the generating polynomial of a set of combinatorial objects, we often seek a hidden structure on those objects in the hope of revealing more information about the polynomial. One desirable structure is one that allows us to generate, preferably in some systematic way, the entire set of combinatorial objects beginning with one initial object that we might regard as indexing the given polynomial. If the polynomial is known or suspected to expand nonnegatively into irreducible characters for some group, then another natural structure for which to search is that of a crystal, the combinatorial skeleton of an unknown module for the group whose character is the given polynomial. 

The polynomials we consider in this paper are the \emph{Kohnert polynomials} introduced by Assaf and Searles \cite{AS19} as the generating polynomials of certain sets of diagrams in the plane. The set of diagrams for a given polynomial can be generated from an initial diagram using \emph{Kohnert moves} \cite{Koh91}, a simple combinatorial rule for moving cells of a diagram down. This process leads to a natural poset structure with a unique maximal element, which we may regard as the indexing diagram for the corresponding Kohnert polynomial.

While the Kohnert poset is advantageous in that one can systematically generate all objects in the set, the poset itself is neither ranked nor is it a lattice. Moreover, given a candidate diagram for the set, the only way to determine if the diagram belongs to the set is to generate the entire set and search for the candidate within it. One of our main results, proven in Theorems~\ref{thm:label-necessary} and \ref{thm:label-sufficient}, gives a static necessary and sufficient condition for a diagram to be in a specific Kohnert poset. This condition generalizes the useful \emph{Kohnert tableaux} defined by Assaf and Searles \cite{AS18} for Demazure characters.

Kohnert polynomials were inspired by two special cases, the geometrically important basis of \emph{Schubert polynomials} and the representation theoretically important basis of \emph{Demazure characters}. Schubert polynomials, introduced by Lascoux and Sch{\"u}tzenberger \cite{LS82}, are polynomial representatives of Schubert classes for the cohomology of the flag manifold whose structure constants, within the ring of polynomials, precisely give the Schubert cell decomposition for the corresponding product of Schubert classes. Demazure modules, introduced by Demazure \cite{Dem74}, form a filtration of highest weight modules compatible with the Bruhat order of the corresponding Weyl group, and their characters \cite{Dem74a} form a basis of the ring of polynomials. Given these deep connections to representation theory and geometry for these instances of Kohnert polynomials, it is natural to ask what other Kohnert polynomials enjoy such connections. Motivated by this question, Assaf and Searles \cite{AS19} characterized diagrams for which the corresponding Kohnert polynomials expand nonnegatively into the \emph{fundamental slide basis} \cite{AS17} for polynomials. Moreover, they give a simple criterion for  diagrams, called \emph{southwest}, for which they conjecture the corresponding Kohnert polynomials expand nonnegatively into Demazure characters.

The main result of this paper, stated in Corollary~\ref{cor:main}, is a proof of this conjecture. Our proof comes via a new structure on the elements of a Kohnert poset, namely that of a \emph{Demazure crystal} \cite{Kas93}. Kashiwara \cite{Kas91} combinatorialized certain highest weight modules through his study of crystal bases, which Littelmann conjectured \cite{Lit95} and Kashiwara proved \cite{Kas93} generalize to Demazure modules via Demazure crystals. Our new crystal operators on diagrams are as simple to define as Kohnert moves, though they are not, in general, given by Kohnert moves. In Theorem~\ref{thm:closed}, we prove the crystal operators act within the Kohnert poset if and, outside a few isolated cases, only if the initial diagram is southwest. The proof of Theorem~\ref{thm:main}, stating the structure is that of a Demazure crystal, comprises the majority of this paper.

The two main tools used in the proof are the aforementioned characterization of Kohnert tableaux and a new algorithm on diagrams, called \emph{rectification}. Rectification is, essentially, the transpose of the crystal operators which acts by pushing cells of a diagram to the left. In Theorem~\ref{thm:commute}, we prove rectification embeds any set of diagrams connected under the crystal operators into a connected highest weight crystal in a way that interwines the crystal operators on diagrams with the Kashiwara crystal operators \cite{Kas91} on the highest weight crystal. In \cite{AG}, Assaf and Gonz\'{a}lez use rectification in this same fashion to embed diagrams that generate specialized nonsymmetric Macdonald polynomials \cite{HHL08} into highest weight crystals, thus realizing a Demazure crystal structure for nonsymmetric Macdonald polynomials. In \cite{AQ}, Assaf and Quijada prove rectification specializes to Robinson--Schensted insertion \cite{Rob38,Sch61} on semistandard Young tableaux and use it as a tool to prove a signed Pieri formula for Demazure characters. We expect this rectification operation to have many more applications in similar contexts.

Our Demazure crystal structure on diagrams is a ranked lattice, though the structure on diagrams for a given Kohnert polynomial is not connected as it was with the Kohnert poset. However, the Demazure crystal operators partition the diagrams in the Kohnert poset into a disjoint union of connected components, each generating a single Demazure character. Thus we prove the nonnegative expansion into Demazure characters. While such an expansion was already known for Schubert polynomials \cite{LS90,RS95,Ass-EG}, we show how using the Kohnert poset structure in conjunction with the Demazure crystal structure leads to more efficient formulas than previously known, both when expanding into Demazure characters and into fundamental slide polynomials. Moreover, the simple nature of the crystal operators on the diagrams in the Kohnert poset suggests a natural module structure for which these Kohnert polynomials are the characters.

This paper is organized as follows. We begin in Section~\ref{sec:poly} with a review of Schubert polynomials, Demazure characters, and Kohnert polynomials, where we state the motivating conjecture for this paper, Conjecture~\ref{conj:demazure}, that Kohnert polynomials indexed by southwest diagrams expand nonnegatively into Demazure characters. Our main tool will be that of crystals, which we review in Section~\ref{sec:crystal}. We state in Definition~\ref{def:KD-raise} our main construction, the crystal operators on Kohnert diagrams. Section~\ref{sec:proof} uses the powerful combinatorial tool of rectification of diagrams to prove our Kohnert crystal embeds into a disjoint union of highest weight crystals. Our final tool is that of labelings of diagrams, developed in Section~\ref{sec:label}, where we prove our main results. We conclude in Section~\ref{sec:applications} by giving explicit formulas for Demazure character and fundamental slide expansions, as well as a characterization of when a Kohnert polynomial is equal to a single Demazure character, parallel to the \emph{vexillary} condition for Schubert polynomials.

%
\section{Polynomials}
%
\label{sec:poly}

Our primary objects of study are geometrically motivated bases of the polynomial ring $\mathbb{Z}[x_1,x_2,\ldots]$ that arise as characters for certain modules.

\subsection{Schubert polynomials}
\label{sec:schubert}

Lascoux and Sch{\"u}tzenberger \cite{LS82} defined polynomial representatives of Schubert classes for the cohomology of the flag manifold with nice algebraic and combinatorial properties using divided difference operators that act on a certain monomial associated to the long permutation according to a reduced expression for the given permutation. 

For a positive integer $i$, the \newword{divided difference operator} $\partial_i$ is the linear operator that acts on polynomials $f \in \mathbb{Z}[x_1,x_2,\ldots]$ by
\begin{equation}
  \partial_i(f) = \frac{f - s_i \cdot f}{x_i - x_{i+1}} ,
\end{equation}
where $s_i \in \Sym_{\infty}$ is the simple transposition that acts on polynomials by exchanging $x_i$ and $x_{i+1}$.

Given a permutation $w\in\Sym_{\infty}$, a \newword{reduced word} $\rho$ for $w$ is a sequence $\rho=(\rho_{\ell},\ldots,\rho_1)$ such that $w = s_{\rho_{\ell}} \cdots s_{\rho_1}$ with $\ell$ minimal. The simple transpositions $s_i$ generate $\Sym_{\infty}$ subject to the relations $s_i^2=1$ and
\begin{itemize}
\item (commutation relation) $s_i s_j = s_j s_i$ for $|i-j|>1$, 
\item (Yang--Baxter relation) $s_i s_{i+1} s_i = s_{i+1} s_i s_{i+1}$.
\end{itemize}
Letting $\Red(w)$ denote the set of reduced words for $w$, any two elements $\rho,\sigma\in\Red(w)$ are connected by a sequence of commutation and Yang--Baxter relations \cite{Tit69}. 

The $\partial_i$ also satisfy the commutation and Yang--Baxter relations along with $\partial_i^2 = 0$. Thus we may define
\begin{eqnarray}
  \partial_w & = & \partial_{\rho_{\ell}} \cdots \partial_{\rho_1} 
\end{eqnarray}
for any reduced word $\rho = (\rho_{\ell},\ldots,\rho_1)$ for $w$.

\begin{definition}[\cite{LS82}]
  Given a permutation $w\in\Sym_n$, the \newword{Schubert polynomial} $\schubert_w$ is given by
  \begin{equation}
    \schubert_w = \partial_{w^{-1} w_0} \left( x_1^{n-1} x_2^{n-2} \cdots x_{n-1} \right),
    \label{e:schubert}
  \end{equation}
  where $w_0 = n \cdots 2 1$ is the longest permutation of $\Sym_n$ of length $\binom{n}{2}$.
  \label{def:schubert}
\end{definition}

The geometric significance of Schubert polynomials was first established by Fulton \cite{Ful92} who made connections between the divided difference operators and modern intersection theory. Surprisingly, at least from their definition, Schubert polynomials form an integral basis for the full polynomial ring, and their structure constants precisely give the Schubert cell decomposition for the corresponding product of Schubert classes. Therefore they give a way to avoid working modulo the ideal of symmetric polynomials in order to compute intersection numbers. 

\subsection{Demazure characters}
\label{sec:demazure}

For a complex, semi-simple Lie algebra $\mathfrak{g}$ with a Cartan subalgebra $\mathfrak{h}$, Demazure \cite{Dem74a} considered the action of a Borel subalgebra $\mathfrak{b} \supset \mathfrak{h}$ on an extremal weight space, thus constructing \emph{Demazure modules}. While the irreducible representations $V^{\lambda}$ of $\mathfrak{g}$ are indexed by dominant weights $\lambda$, the corresponding Demazure modules $V^{\lambda}_w$ are index by a pair $(\lambda,w)$ where $\lambda$ is a dominant weight and $w$ is an element of the Weyl group. In the case of $\mathfrak{gl}_n$, $\lambda$ is a partition and $w$ is a permutation, which is equivalent to the weak composition $a = w \cdot \lambda$.

Demazure generalized the Weyl character formula \cite{Dem74} to these Demazure modules to obtain \emph{Demazure characters}. His formula can be stated in terms of degree-preserving divided difference operators.

The \newword{degree-preserving divided difference operator} $\pi_i$ is the linear operator that acts on polynomials $f \in \mathbb{Z}[x_1,x_2,\ldots]$ by
\begin{equation}
  \pi_i(f) = \partial_i (x_i f) .
\end{equation}
These $\pi_i$ also satisfy the commutation and Yang--Baxter relations along with $\pi_i^2 = \pi$, allowing us to define
\begin{eqnarray}
  \pi_w & = & \pi_{\rho_{\ell}} \cdots \pi_{\rho_1}
\end{eqnarray}
for any reduced word $\rho = (\rho_{\ell},\ldots,\rho_1)$ for $w$.

\begin{definition}[\cite{Dem74}]
  Given a weak composition $\comp{a}$, the \newword{Demazure character} $\key_{\comp{a}}$ is given by
  \begin{equation}
    \key_{\comp{a}} = \pi_{w} \left( x_1^{\lambda_1} x_2^{\lambda_2} \cdots x_{n}^{\lambda_{\ell}} \right),
    \label{e:key}
  \end{equation}
  where $\lambda=\mathrm{sort}(\comp{a})$ is the unique partition in the $\Sym_{\infty}$ orbit of $\comp{a}$, and $w$ is the unique minimal length permutation such that $w \cdot \lambda = \comp{a}$.
  \label{def:key}
\end{definition}

These Demazure characters form another basis of the polynomial ring, and have been studied under the name \emph{standard bases} by Lascoux and Sch{\"u}tzenberger \cite{LS90} and Kohnert \cite{Koh91}, and under the name \emph{key polynomials} by Reiner and Shimozono \cite{RS95}, Mason \cite{Mas09}, Assaf and Searles \cite{AS18}, and others.

Lascoux and Sch{\"u}tzenberger \cite{LS90} noticed coincidences with Schubert polynomials and Demazure characters giving rise to a characterization of \newword{vexillary permutations} $w$ for which $\schubert_w = \key_{\comp{a}}$ for some $\comp{a}$. Moreover, Lascoux and Sch{\"u}tzenberger \cite{LS90} give a formula for the nonnegative expansion of a Schubert polynomial into Demazure characters. Their formula, with proof details supplied by Reiner and Shimozono \cite[Theorem~4]{RS95}, relies on the combinatorial formula due to Edelman and Greene \cite{EG87} for the Schur expansion for Stanley symmetric functions \cite{Sta84}. Each term in the Schur expansion corresponds to an \newword{increasing reduced word}, that is, a reduced word that occurs as the row reading word of an increasing Young tableau. The Demazure expansion for Schubert polynomials as stated by Lascoux and Sch{\"u}tzenberger \cite{LS90} is determined by searching the Coxeter--Knuth equivalence class of each increasing reduced word for the word with minimal weak descent composition. Assaf \cite{Ass-EG} gives a more explicit formula for this expansion by \emph{lifting} increasing reduced words using crystal raising operators.

\begin{theorem}[\cite{LS90,RS95,Ass-EG}]
  For $w$ a permutation, we have
  \begin{equation}
    \schubert_{w} = \sum_{\substack{\rho \in\Red(w) \\ \rho = \mathrm{row}(T)}} \key_{\des(\mathrm{lift}(\rho))},
    \label{e:nilkey}
  \end{equation}
  where the sum is over reduced words that occur as reading words for increasing Young tableaux, and $\des(\rho)$ is the weak descent composition of $\rho$. 
  \label{thm:LS}
\end{theorem}

For example, Fig.~\ref{fig:yam} shows the increasing reduce words and their lifts for the permutation $13625847$. Thus we compute the Demazure expansion of the Schubert polynomial $\schubert_{13625847}$ from Fig.~\ref{fig:yam}, giving
\[ \schubert_{13625847} = \key_{(0,1,3,0,1,2)} + \key_{(0,2,3,0,0,2)} + \key_{(0,3,3,0,0,1)} + \key_{(0,1,4,0,1,1)} + \key_{(0,2,4,0,0,1)} . \]

\begin{figure}[ht]
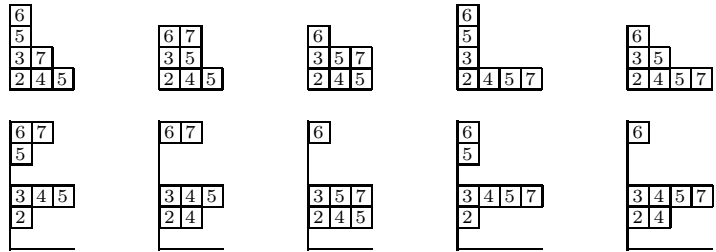

  \begin{displaymath}
    \arraycolsep=2\cellsize
    \begin{array}{ccccc}
      \tableau{6 \\ 5 \\ 3 & 7 \\ 2 & 4 & 5} &
      \tableau{\\ 6 & 7 \\ 3 & 5 \\ 2 & 4 & 5} &
      \tableau{\\ 6 \\ 3 & 5 & 7 \\ 2 & 4 & 5} &
      \tableau{6 \\ 5 \\ 3 \\ 2 & 4 & 5 & 7} &
      \tableau{\\ 6 \\ 3 & 5 \\ 2 & 4 & 5 & 7} \\ \\
      \vline\tableau{6 & 7 \\ 5 \\ \\ 3 & 4 & 5 \\ 2 \\ \\\hline} &
      \vline\tableau{6 & 7 \\ \\ \\ 3 & 4 & 5 \\ 2 & 4 \\ \\\hline} &
      \vline\tableau{6 \\ \\ \\ 3 & 5 & 7 \\ 2 & 4 & 5 \\ \\\hline} &
      \vline\tableau{6 \\ 5 \\ \\ 3 & 4 & 5 & 7 \\ 2 \\ \\\hline} &
      \vline\tableau{6 \\ \\ \\ 3 & 4 & 5 & 7 \\ 2 & 4 \\ \\\hline} 
    \end{array}
  \end{displaymath}
   \caption{\label{fig:yam}The set of increasing reduced words for $13625847$ (top) drawn as increasing Young tableaux, and their lifts (bottom) drawn as weak descent tableaux.}
\end{figure}

This result can be understood representation theoretically with Schubert functors defined by Kra\'{s}kiewicz and Pragacz \cite{KP87} used to construct Demazure modules whose characters are Schubert polynomials.

\subsection{Kohnert polynomials}
\label{sec:kohnert}

A \newword{diagram} is a finite collection of points, called the cells of the diagram, in the first quadrant of $\mathbb{Z}\times\mathbb{Z}$. We draw each cell $(c,r)$ of a diagram as the circle inscribed in the unit cell with northeastern coordinate $(c,r)$. We refer to columns and rows as the $x$- and $y$-coordinates, respectively.

Kohnert's combinatorial model for Demazure characters \cite{Koh91} is in terms of diagrams, specifically those generated by the following operation.

\begin{definition}[\cite{Koh91}]
  A \newword{Kohnert move} on a diagram selects the rightmost cell of a given row and moves the cell down within its column to the first available position below, if it exists, jumping over other cells in its way as needed. 
\label{def:kohnert_move}
\end{definition}

Given a diagram $D$, let $\KD(D)$ denote the set of diagrams that can be obtained by some sequence of Kohnert moves on $D$. Note that there might be multiple ways to obtain a diagram from different Kohnert moves of a given diagram, but each resulting diagram is included in the set exactly once.

For $D$ a diagram and $T\in\KD(D)$, notice the corresponding sets of Kohnert diagrams are nested $\KD(T) \subseteq \KD(D)$. Thus we may consider the partial order on diagrams in $\KD(D)$ given by the transitive closure of relations $T \prec S$ if $T$ can be obtained from $S$ by a Kohnert move. This partial order has a unique maximal element, namely $D$, but is not, in general, ranked nor does it have a unique minimal element. For example, see Fig.~\ref{fig:kohnert-poset}.

\begin{figure}[ht]
  \begin{center}
    \begin{tikzpicture}[xscale=1.7,yscale=1.7]
      \node at (3,6)     (A) {$\cirtab{ & & \ \\ \ & & \\ \ & \ & \ \\ & & \\\hline}$};
      \node at (3.5,5)   (B) {$\cirtab{ & & \\ \ & & \ \\ \ & \ & \ \\ & & \\\hline}$};
      \node at (2,5)     (C) {$\cirtab{ & & \ \\ & & \\ \ & \ & \ \\ \ & & \\\hline}$};
      \node at (5,5)     (D) {$\cirtab{ & & \ \\ \ & & \\ \ & \ & \\ & & \ \\\hline}$};
      \node at (4,4)     (E) {$\cirtab{ & & \\ \ & & \ \\ \ & \ & \\ & & \ \\\hline}$};
      \node at (1,4)     (F) {$\cirtab{ & & \\ & & \ \\ \ & \ & \ \\ \ & & \\\hline}$};
      \node at (2.75,4)   (G) {$\cirtab{ & & \ \\ & & \\ \ & \ & \\ \ & & \ \\\hline}$};
      \node at (5,4)     (H) {$\cirtab{ & & \ \\ \ & & \\ \ & & \\ & \ & \ \\\hline}$};
      \node at (6,3)     (I) {$\cirtab{ & & \\ \ & & \ \\ \ & & \\ & \ & \ \\\hline}$};
      \node at (3,3)     (J) {$\cirtab{ & & \\ \ & & \\ \ & \ & \ \\ & & \ \\\hline}$};
      \node at (2,3)     (K) {$\cirtab{ & & \\ & & \ \\ \ & \ & \\ \ & & \ \\\hline}$};
      \node at (5,3)   (L) {$\cirtab{ & & \ \\ \ & & \\ & & \\ \ & \ & \ \\\hline}$};
      \node at (5,2)     (M) {$\cirtab{ & & \\ \ & & \ \\ & & \\ \ & \ & \ \\\hline}$};
      \node at (6,1.5)   (N) {$\cirtab{ & & \\ \ & & \\ \ & & \ \\ & \ & \ \\\hline}$};
      \node at (1.5,1.5) (O) {$\cirtab{ & & \\ & & \\ \ & \ & \ \\ \ & & \ \\\hline}$};
      \node at (4,2)   (P) {$\cirtab{ & & \ \\ & & \\ \ & & \\ \ & \ & \ \\\hline}$};
      \node at (4.5,1)   (Q) {$\cirtab{ & & \\ \ & & \\ & & \ \\ \ & \ & \ \\\hline}$};
      \node at (3,1)     (R) {$\cirtab{ & & \\ & & \ \\ \ & & \\ \ & \ & \ \\\hline}$};
      \node at (4,0)     (S) {$\cirtab{ & & \\ & & \\ \ & & \ \\ \ & \ & \ \\\hline}$};
      \draw[thin] (A) -- (B) ;
      \draw[thin] (A) -- (C) ;
      \draw[thin] (A) -- (D) ;
      \draw[thin] (B) -- (E) ;
      \draw[thin] (B) -- (J) ;
      \draw[thin] (C) -- (F) ;
      \draw[thin] (C) -- (G) ;
      \draw[thin] (D) -- (E) ;
      \draw[thin] (D) -- (H) ;
      \draw[thin] (E) -- (I) ;
      \draw[thin] (E) -- (J) ;
      \draw[thin] (F) -- (O) ;
      \draw[thin] (F) -- (K) ;
      \draw[thin] (G) -- (K) ;
      \draw[thin] (G) -- (P) ;
      \draw[thin] (H) -- (I) ;
      \draw[thin] (H) -- (P) ;
      \draw[thin] (H) -- (L) ;
      \draw[thin] (I) -- (M) ;
      \draw[thin] (I) -- (N) ;
      \draw[thin] (J) -- (O) ;
      \draw[thin] (K) -- (O) ;
      \draw[thin] (K) -- (R) ;
      \draw[thin] (L) -- (M) ;
      \draw[thin] (L) -- (P) ;
      \draw[thin] (M) -- (Q) ;
      \draw[thin] (N) -- (S) ;
      \draw[thin] (P) -- (R) ;
      \draw[thin] (Q) -- (S) ;
      \draw[thin] (R) -- (S) ;
    \end{tikzpicture}
    \caption{\label{fig:kohnert-poset}The poset of Kohnert moves on Kohnert diagrams for the top diagram.}
  \end{center}
\end{figure}
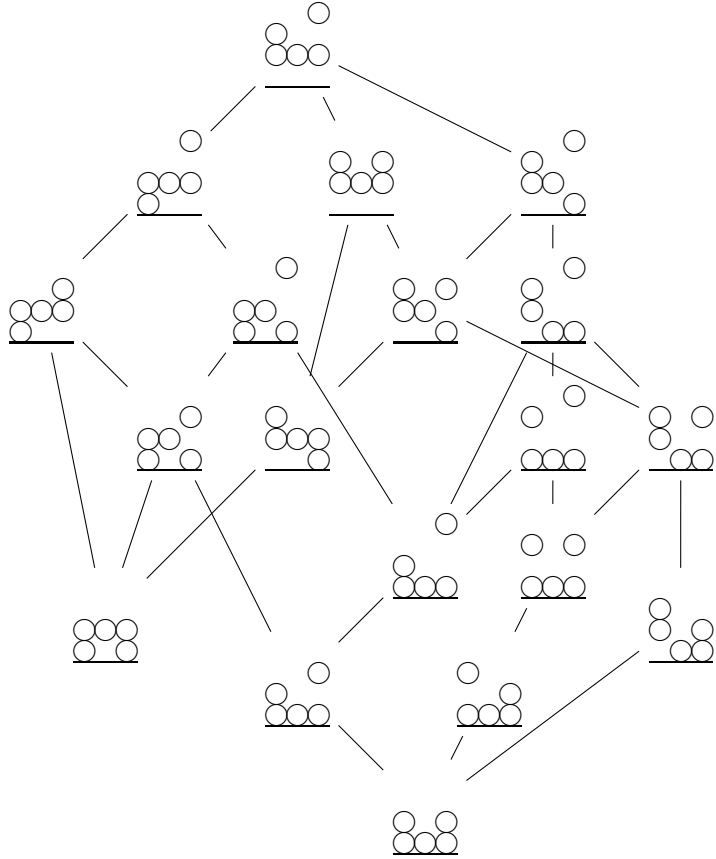

The \newword{weight} of a diagram $D$, denoted by $\wt(D)$, is the weak composition whose $i$th part is the number of cells of $D$ in row $i$.

Kohnert's model for Demazure characters began with \newword{composition diagrams}, the unique left-justified diagrams for each given weight. He proved the following.

\begin{theorem}[\cite{Koh91}]
  For a weak composition $\comp{a}$, the Demazure character $\key_{\comp{a}}$ is given by
  \[ \key_{\comp{a}} = \sum_{T \in \KD(\D(\comp{a}))} x_1^{\wt(T)_1} \cdots x_n^{\wt(T)_n}, \]
  where $\D(\comp{a})$ denotes the composition diagram of weight $\comp{a}$.
  \label{thm:kohnert}
\end{theorem}

To each permutation $w$, we may associate the \newword{Rothe diagram} given by
\begin{equation}
  \D(w) = \{ (i,w_j) \mid i<j \mbox{ and } w_i > w_j \} .
  \label{e:rothe}
\end{equation}
Visually, we may write the word for $w$ along the $y$-axis, then place a cell in row $i$, column $w_j$ when $w_i>w_j$ for $w_j$ above row $i$. For example, see Fig.~\ref{fig:rothe}.

\begin{figure}[ht]
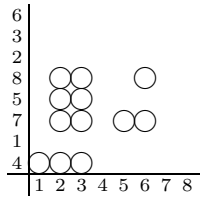

  \begin{displaymath}
    \nulltab{6 \\ 3 \\ 2 \\ 8 \\ 5 \\ 7 \\ 1 \\ 4 \\\hline }\hss%
    \vline\nulltab{ & \\ & \\ & \\ 
      & \circify{\ } & \circify{\ } & & & \circify{\ } \\
      & \circify{\ } & \circify{\ } \\
      & \circify{\ } & \circify{\ } & & \circify{\ } & \circify{\ } \\
      & \\
      \circify{\ } & \circify{\ } & \circify{\ } \\\hline 
      1 & 2 & 3 & 4 & 5 & 6 & 7 & 8}
  \end{displaymath}
   \caption{\label{fig:rothe}The Rothe diagram $\D(w)$ for the permutation $w = 41758236$.}
\end{figure}

Kohnert asserted Kohnert moves on the Rothe diagram $\D(w)$ generate the Schubert polynomial $\schubert_w$, giving the first conjectured combinatorial formula for Schubert polynomials. Winkel \cite{Win99,Win02} gives two proofs, the second an improvement on the first after it faced broad criticism, though neither the original nor revised proof is widely accepted given the intricate and opaque nature of the arguments. In \cite{Ass-KR}, we give a direct, bijective proof of Kohnert's rule for Schubert polynomials.

\begin{theorem}[\cite{Ass-KR}]
  For a permutation $w$, the Schubert polynomial $\schubert_{w}$ is given by
  \[ \schubert_{w} = \sum_{T \in \KD(\D(w))} x_1^{\wt(T)_1} \cdots x_n^{\wt(T)_n}, \]
  where $\D(w)$ denotes the Rothe diagram of $w$.
  \label{thm:kohnertrule}
\end{theorem}

Assaf and Searles \cite[Definition~2.2]{AS19} considered applying Kohnert's model in greater generality, creating more general \emph{Kohnert polynomials} as defined below.

\begin{definition}[\cite{AS19}]
    For a diagram $D$, the \newword{Kohnert polynomial indexed by $D$} is 
  \begin{equation}
    \kohnert_{D} = \sum_{T \in \KD(D)} x_1^{\wt(T)_1} \cdots x_n^{\wt(T)_n}.
  \end{equation}
  \label{def:kohnert_poly}
\end{definition}

Assaf and Searles \cite[Theorem~3.7]{AS19} showed Kohnert polynomials expand nonnegatively into the monomial slide polynomial basis \cite[Definition~3.3]{AS17}, and they \cite[Theorem~4.14]{AS19} characterized those diagrams for which the Kohnert polynomials expand nonnegatively into the fundamental slide polynomial basis \cite[Definition~3.6]{AS17}. Furthermore, they conjectured a characterization for diagrams whose corresponding Kohnert polynomials expand nonnegatively into the key polynomial basis.

\begin{definition}[\cite{AS19}]
  A diagram $D$ is \newword{southwest} if for every pair of cells $(c_1,r_2)$, $(c_2,r_1)$ in $D$ with $r_1<r_2$ and $c_1<c_2$, the cell $(c_1,r_1)$ is also in $D$.
  \label{def:diagram-sw}
\end{definition}

Graphically, for any two cells positioned with one strictly northwest of the other, then there must a cell at their southwest corner, i.e. in the column of the northwestern cell and the row of the southeastern cell. 

\begin{conjecture}[\cite{AS19}]
  Given a southwest diagram $D$, the Kohnert polynomial $\kohnert_D$ expands non-negatively into Demazure characters.
  \label{conj:demazure}
\end{conjecture}

Left-justified diagrams, which by Theorem~\ref{thm:kohnert} generate Demazure characters, trivially satisfy the southwest condition as do Rothe diagrams, which by Theorem~\ref{thm:kohnertrule} generate Schubert polynomials. The diagram $D$ at the top of Fig.~\ref{fig:kohnert-poset} is also southwest, and expanding the Kohnert polynomial gives
\[ \kohnert_D = \key_{(0,3,2)} + \key_{(0,3,1,1)} . \]

We prove Conjecture~\ref{conj:demazure} by constructing for each southwest diagram $D$ a Demazure crystal whose character is the Kohnert polynomial $\kohnert_D$.

%
\section{Crystals}
%
\label{sec:crystal}

We harness the power of crystal graphs, certain combinatorial skeletons of representations, to prove positivity results for polynomials that arise as characters.

\subsection{Tableaux crystals}
\label{sec:crystal-tab}

Kashiwara \cite{Kas90,Kas91} introduced crystal bases in his study of the representation theory of quantized universal enveloping algebra $U_q(\mathfrak{g})$ for the Lie algebra $\mathfrak{g}$.
The character of a crystal coincides with the character of the representation, and so the existence of a crystal structure implies a positive expansion for the character in terms of the basis of irreducible characters.


A finite \newword{crystal} consists of the following data: a nonempty, finite set $\B$, a \newword{weight map} $\wt:\B \rightarrow \mathbb{N}^n$, and \newword{raising} and \newword{lowering operators} $\e_i, \f_i : \B \rightarrow \B \cup \{0\}$, for $1 \leq i < n$, satisfying $\e_i(b) = b'$ if and only if $\f_i(b')=b$ and in this case $\wt(b') = \wt(b) + \alpha_i$ , where $\alpha_i \in \mathbb{N}^n$ is the \newword{simple root} with $1$ in position $i$, $-1$ in position $i+1$, and $0$ elsewhere.

The \newword{crystal graph} for $\B$ is the directed, colored graph with vertex set $\B$ and a directed $i$-edge from $b$ to $\f_i(b)$ if $\f_i(b)\neq 0$, where all edges to $0$ are omitted.

Using the weight map, we define the \newword{character} of the crystal $\B$ by
\begin{equation}
  \mathrm{ch}(\B) = \sum_{b \in \B} x_1^{\wt(b)_1} x_2^{\wt(b)_2} \cdots x_{n}^{\wt(b)_{n}}.
  \label{e:char}
\end{equation}

Given a crystal $\B$, an element $b \in \B$ is a \newword{highest weight element} if $e_i(b)=0$ for all $i \geq 1$. For $\B$ connected, if $b$ is a highest weight element, then $\wt(b)$ is a dominant weight called the \newword{highest weight} of $\B$. 

Connected highest weight crystals are in one-to-one correspondence with irreducible representations of $U_q(\mathfrak{gl}_n)$, which are naturally indexed by dominant weights, or \emph{partitions} $\lambda$. The corresponding crystal bases $\B(\lambda)$ are naturally indexed by \emph{semistandard Young tableaux} of shape $\lambda$. Kashiwara and Nakashima \cite{KN94} and, independently, Littelmann \cite{Lit95} give an explicit combinatorial construction of the crystal graph on semistandard Young tableaux. 

The \newword{Young diagram} of a partition $\lambda$ has $\lambda_i$ left justified unit cells in row $i$. To differentiate between Young diagrams and diagrams used to define Kohnert polynomials, we draw square cells for the latter.

\begin{definition}
  For a partition $\lambda$, a \newword{semistandard Young tableau of shape $\lambda$} is a filling of the cells of $\lambda$ with positive integers such that entries weakly increase left to right along rows and strictly increase bottom to top along columns. 
  \label{def:SSYT}
\end{definition}

Given a partition $\lambda$, we let $\SSYT_n(\lambda)$ denote the set of semistandard Young tableaux of shape $\lambda$ with entries in $\{1,2,\ldots,n\}$. This serves as the underlying set for the crystal basis $\B(\lambda)$. We define the weight map by setting $\wt(T)_i$ to be the number of cells of $T$ with entry equal to $i$. 

\begin{definition}[\cite{KN94,Lit95}]
  For $T\in\SSYT_n(\lambda)$ and $1 \leq i < n$, the \newword{$i$-pairing} of cells of $T$ containing entries $i$ or $i+1$ is defined as follows:
  \begin{itemize}
  \item $i$-pair cells containing $i$ and $i+1$ whenever they appear in the same column,
  \item iteratively $i$-pair an unpaired $i+1$ with an unpaired $i$ to its right whenever all entries $i$ and $i+1$ that lie between are already $i$-paired.
  \end{itemize}
  \label{def:SSYT-pair}
\end{definition}



\begin{definition}[\cite{KN94,Lit95}]
  For $T\in\SSYT_n(\lambda)$ and $1 \leq i < n$, the \newword{lowering operator} $\f_i$ acts on $T$ as follows: if $T$ has no unpaired entries $i$, then $\f_i(T)=0$; else, change the rightmost unpaired $i$ to $i+1$ leaving all other entries unchanged.
  \label{def:SSYT-lower}
\end{definition}

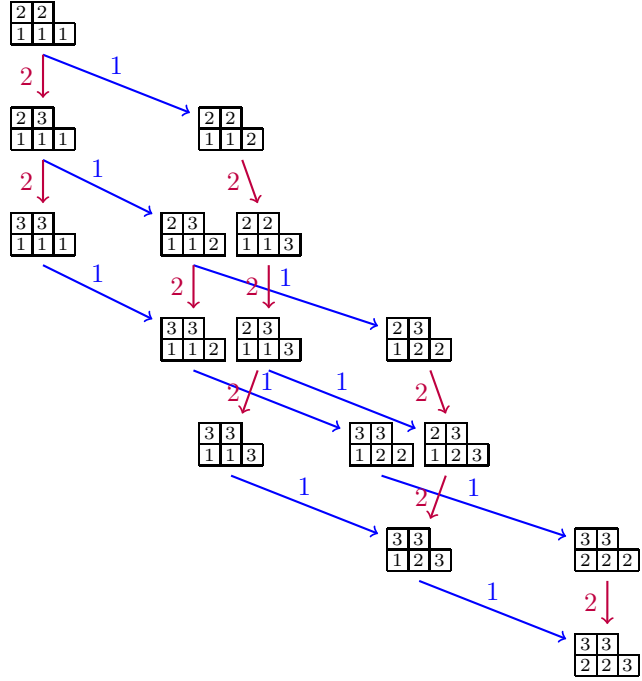
\begin{figure}[ht]
  \begin{center}
    \begin{tikzpicture}[xscale=2.5,yscale=1.4]
      \node at (2,6)  (A) {$\tableau{ 2 & 2 \\ 1 & 1 & 1 \\\hline}$};
      \node at (2,5)  (B) {$\tableau{ 2 & 3 \\ 1 & 1 & 1 \\\hline}$};
      \node at (3,5)  (C) {$\tableau{ 2 & 2 \\ 1 & 1 & 2 \\\hline}$};
      \node at (2,4)  (D) {$\tableau{ 3 & 3 \\ 1 & 1 & 1 \\\hline}$};
      \node at (3.2,4)(E) {$\tableau{ 2 & 2 \\ 1 & 1 & 3 \\\hline}$};
      \node at (2.8,4)(F) {$\tableau{ 2 & 3 \\ 1 & 1 & 2 \\\hline}$};
      \node at (3.2,3)(G) {$\tableau{ 2 & 3 \\ 1 & 1 & 3 \\\hline}$};
      \node at (2.8,3)(H) {$\tableau{ 3 & 3 \\ 1 & 1 & 2 \\\hline}$};
      \node at (4,3)  (I){ $\tableau{ 2 & 3 \\ 1 & 2 & 2 \\\hline}$};
      \node at (3,2)  (J) {$\tableau{ 3 & 3 \\ 1 & 1 & 3 \\\hline}$};
      \node at (4.2,2)(K) {$\tableau{ 2 & 3 \\ 1 & 2 & 3 \\\hline}$};
      \node at (3.8,2)(L) {$\tableau{ 3 & 3 \\ 1 & 2 & 2 \\\hline}$};
      \node at (4,1)  (M) {$\tableau{ 3 & 3 \\ 1 & 2 & 3 \\\hline}$};
      \node at (5,1)  (N) {$\tableau{ 3 & 3 \\ 2 & 2 & 2 \\\hline}$};
      \node at (5,0)  (O) {$\tableau{ 3 & 3 \\ 2 & 2 & 3 \\\hline}$};
      \draw[thick,blue  ,->](A.south) -- (C)   node[midway,above]{$1$};
      \draw[thick,blue  ,->](B.south) -- (F)   node[midway,above]{$1$};
      \draw[thick,blue  ,->](D.south) -- (H)   node[midway,above]{$1$};
      \draw[thick,blue  ,->](F.south) -- (I)   node[midway,above]{$1$};
      \draw[thick,blue  ,->](G.south) -- (K)   node[midway,above]{$1$};
      \draw[thick,blue  ,->](H.south) -- (L)   node[midway,above]{$1$};
      \draw[thick,blue  ,->](J.south) -- (M)   node[midway,above]{$1$};
      \draw[thick,blue  ,->](L.south) -- (N)   node[midway,above]{$1$};
      \draw[thick,blue  ,->](M.south) -- (O)   node[midway,above]{$1$};
      \draw[thick,purple,->](A) -- (B)   node[midway,left ]{$2$};
      \draw[thick,purple,->](B) -- (D)   node[midway,left ]{$2$};
      \draw[thick,purple,->](C) -- (E)   node[midway,left ]{$2$};
      \draw[thick,purple,->](E) -- (G)   node[midway,left ]{$2$};
      \draw[thick,purple,->](F) -- (H)   node[midway,left ]{$2$};
      \draw[thick,purple,->](G) -- (J)   node[midway,left ]{$2$};
      \draw[thick,purple,->](I) -- (K)   node[midway,left ]{$2$};
      \draw[thick,purple,->](K) -- (M)   node[midway,left ]{$2$};
      \draw[thick,purple,->](N) -- (O)   node[midway,left ]{$2$};
    \end{tikzpicture}
    \caption{\label{fig:crystal}The crystal $\B(3,2,0)$ on $\SSYT_3(3,2)$, with crystal edges $\f_1 {\color{blue}\searrow}$, and $\f_2 {\color{purple}\downarrow}$.}
  \end{center}
\end{figure}

This implicitly defines the raising operators $\e_i$ as well by the property that for $T,T'\in\SSYT(\lambda)$, we have $\e_i(T) = T'$ if and only if $\f_i(T')=T$. See Fig.~\ref{fig:crystal}.

\begin{theorem}[\cite{KN94,Lit95}]
  The data $(\SSYT_n(\lambda),\wt,\{\e_i,\f_i\}_{1 \leq i < n})$ determines the irreducible, highest weight crystal $\B(\lambda)$ with highest weight $\lambda$.
\end{theorem}

\subsection{Demazure crystals}
\label{sec:crystal-dem}

Littelmann \cite{Lit95} conjectured a crystal structure for Demazure modules as certain truncations of highest weight crystals. Kashiwara \cite{Kas93} proved the result, giving a new proof of the Demazure character formula.

Given a subset $X \subseteq \B(\lambda)$, define \newword{Demazure operators} $\mathfrak{D}_i$ by
\begin{equation}
  \mathfrak{D}_i (X) = \{ b \in \B(\lambda) \mid \e_i^k(b) \in X \mbox{ for some } k \geq 0 \},
  \label{e:D}
\end{equation}
where $\e_i$ denotes the raising operator for $\B$. As with the divided difference operators on polynomials, these operators on crystals satisfy the commutation and Yang--Baxter relations for the symmetric group, and so we may define
\begin{equation}
  \mathfrak{D}_w = \mathfrak{D}_{\rho_\ell} \cdots \mathfrak{D}_{\rho_1}
  \label{e:Dw}
\end{equation}
for any reduced word $\rho = (\rho_\ell,\ldots,\rho_1)$ for the permutation $w$.

\begin{definition}[\cite{Lit95}]
  For $\lambda$ a partition of length $n$ and $w$ a permutation of $\mathcal{S}_n$, the \newword{Demazure crystal} $\B_w(\lambda)$ is given by
  \begin{equation}
    \B_w(\lambda) = \mathfrak{D}_{w} \{ u_\lambda \},
    \label{e:BwL}
  \end{equation}
  where $u_\lambda$ is the highest weight element in $\B(\lambda)$.
  \label{def:BwL}
\end{definition}

Notice the crystal $\B_{\mathrm{id}}(\lambda)$ consists solely of the highest weight element, and, for $w_0^{(n)}$ the longest element of $\Sym_n$, we have $\B_{w_0^{(n)}}(\lambda) = \B(\lambda)$. Fig.~\ref{fig:demazure-crystal} shows the Demazure crystals $\B_{123}(3,2,0)$, $\B_{132}(3,2,0)$ and $\B_{312}(3,2,0)$, which can be constructed from the crystal $\B_{321}(3,2,0) = \B(3,2,0)$ shown in Fig.~\ref{fig:crystal}.

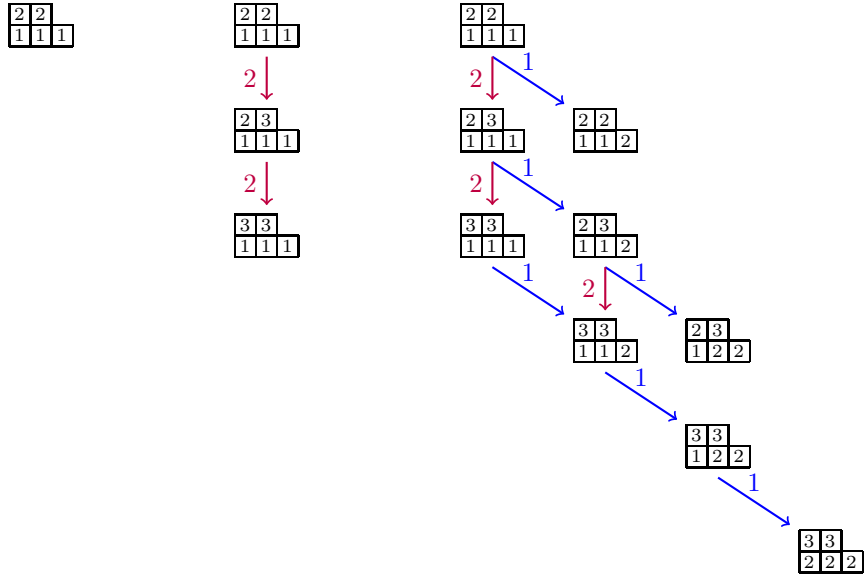
\begin{figure}[ht]
  \begin{center}
    \begin{tikzpicture}[xscale=1.5,yscale=1.4]
      \node at (-1,6) (AAA){$\tableau{ 2 & 2 \\ 1 & 1 & 1 \\\hline}$};
      \node at (1,6)  (AA) {$\tableau{ 2 & 2 \\ 1 & 1 & 1 \\\hline}$};
      \node at (1,5)  (BB) {$\tableau{ 2 & 3 \\ 1 & 1 & 1 \\\hline}$};
      \node at (1,4)  (DD) {$\tableau{ 3 & 3 \\ 1 & 1 & 1 \\\hline}$};
      \draw[thick,purple,->](AA) -- (BB)   node[midway,left ]{$2$};
      \draw[thick,purple,->](BB) -- (DD)   node[midway,left ]{$2$};
      \node at (3,6)  (A) {$\tableau{ 2 & 2 \\ 1 & 1 & 1 \\\hline}$};
      \node at (3,5)  (B) {$\tableau{ 2 & 3 \\ 1 & 1 & 1 \\\hline}$};
      \node at (4,5)  (C) {$\tableau{ 2 & 2 \\ 1 & 1 & 2 \\\hline}$};
      \node at (3,4)  (D) {$\tableau{ 3 & 3 \\ 1 & 1 & 1 \\\hline}$};
      \node at (4,4)(F)   {$\tableau{ 2 & 3 \\ 1 & 1 & 2 \\\hline}$};
      \node at (4,3)(H)   {$\tableau{ 3 & 3 \\ 1 & 1 & 2 \\\hline}$};
      \node at (5,3)  (I) {$\tableau{ 2 & 3 \\ 1 & 2 & 2 \\\hline}$};
      \node at (5,2)(L)   {$\tableau{ 3 & 3 \\ 1 & 2 & 2 \\\hline}$};
      \node at (6,1)  (N) {$\tableau{ 3 & 3 \\ 2 & 2 & 2 \\\hline}$};
      \draw[thick,blue  ,->](A.south) -- (C)   node[midway,above]{$1$};
      \draw[thick,blue  ,->](B.south) -- (F)   node[midway,above]{$1$};
      \draw[thick,blue  ,->](D.south) -- (H)   node[midway,above]{$1$};
      \draw[thick,blue  ,->](F.south) -- (I)   node[midway,above]{$1$};
      \draw[thick,blue  ,->](H.south) -- (L)   node[midway,above]{$1$};
      \draw[thick,blue  ,->](L.south) -- (N)   node[midway,above]{$1$};
      \draw[thick,purple,->](A) -- (B)   node[midway,left ]{$2$};
      \draw[thick,purple,->](B) -- (D)   node[midway,left ]{$2$};
      \draw[thick,purple,->](F) -- (H)   node[midway,left ]{$2$};
    \end{tikzpicture}
    \caption{\label{fig:demazure-crystal}The Demazure crystals $\B_{123}(3,2,0)$ (left), $\B_{132}(3,2,0)$ (middle) and $\B_{312}(3,2,0)$ (right) on semistandard Young tableaux with edges $\f_1 {\color{blue}\searrow}$, $\f_2 {\color{purple}\downarrow}$.}
  \end{center}
\end{figure}

\begin{theorem}[\cite{Kas93}]
  The character of the Demazure crystal $\B_w(\lambda)$ is the Demazure character $\key_{w \cdot \lambda}$.
\end{theorem}


Generalizing the tableaux crystals, Assaf and Schilling \cite[Definition~3.7]{ASc18} defined an explicit Demazure crystal structure on semistandard \emph{key tableaux} \cite{Ass18}, objects that correspond to Mason's semi-skyline augmented fillings \cite{Mas09}.

\begin{definition}[\cite{Ass18}]
  For a weak composition $\comp{a}$, a \newword{semistandard key tableau of shape $\comp{a}$} is a filling of the cells of the composition diagram for $\comp{a}$ with positive integers such that
  \begin{enumerate}
  \item entries weakly decrease left to right along rows, 
  \item if $i$ is above $k$ in the same column with $i<k$, then there exists $j>i$ to the right of $k$ in the same row, and
  \item entries in row $i$ are at most $i$.
  \end{enumerate}
  \label{def:SSKT}
\end{definition}

Given a weak composition $\comp{a}$, we let $\SSKT(\comp{a})$ denote the set of semistandard key tableaux of shape $\comp{a}$. Notice by the third condition of Definition~\ref{def:SSKT}, we do not need to restrict the values of the entries.

This serves as the underlying set for the crystal basis $\B_w(\lambda)$, where $\comp{a} = w \cdot \lambda$. We define the weight map as with tableaux by setting $\wt(T)_i$ to be the number of cells of $T$ with entry equal to $i$. We alter the $i$-pairing rule from Definition~\ref{def:SSYT-pair} to account for the change in rows from increasing to decreasing.

\begin{definition}[\cite{ASc18}]
  For $T\in\SSKT(\comp{a})$ and $1 \leq i < n$, the \newword{$i$-pairing} of cells of $T$ containing entries $i$ or $i+1$ is defined as follows:
  \begin{itemize}
  \item $i$-pair cells containing $i$ and $i+1$ whenever they appear in the same column,
  \item iteratively $i$-pair an unpaired $i$ with an unpaired $i+1$ to its right whenever all entries $i$ and $i+1$ that lie between are already $i$-paired.
  \end{itemize}
  \label{def:SSKT-pair}
\end{definition}

Our presentation of the pairing rule and raising operators is taken from \cite{AG} where it is shown to be equivalent to that in \cite{ASc18}.

\begin{definition}[\cite{ASc18}]
  For $T\in\SSKT(\comp{a})$ and $1 \leq i < n$, the \newword{raising operator} $\e_i$ acts on $T$ as follows: if $T$ does not have any unpaired entry $i+1$, then $e_i(T)=0$; otherwise, $e_i$ changes the rightmost unpaired $i+1$ to $i$ and swaps the entries $i$ and $i+1$ in each of the consecutive columns left of this entry that have an $i+1$ in the same row and an $i$ above.
  \label{def:SSKT-raise}
\end{definition}

\begin{figure}[ht]
  \begin{center}
    \begin{tikzpicture}[xscale=1.5,yscale=1.4]
      \node at (-1,6)  (AAA) {$\vline\tableau{ \\ 2 & 2 \\ 1 & 1 & 1 \\\hline}$};
      \node at (1,6)  (AA)   {$\vline\tableau{ 2 & 2 \\ \\ 1 & 1 & 1 \\\hline}$};
      \node at (1,5)  (BB)   {$\vline\tableau{ 3 & 2 \\ \\ 1 & 1 & 1 \\\hline}$};
      \node at (1,4)  (DD)   {$\vline\tableau{ 3 & 3 \\ \\ 1 & 1 & 1 \\\hline}$};
      \draw[thick,purple,->](AA) -- (BB)   node[midway,left ]{$2$};
      \draw[thick,purple,->](BB) -- (DD)   node[midway,left ]{$2$};
      \node at (3,6)  (A) {$\vline\tableau{ 2 & 2 \\ 1 & 1 & 1 \\ & \\\hline}$};
      \node at (3,5)  (B) {$\vline\tableau{ 3 & 2 \\ 1 & 1 & 1 \\ & \\\hline}$};
      \node at (4,5)  (C) {$\vline\tableau{ 1 & 1 \\ 2 & 2 & 2 \\ & \\\hline}$};
      \node at (3,4)  (D) {$\vline\tableau{ 3 & 3 \\ 1 & 1 & 1 \\ & \\\hline}$};
      \node at (4,4)  (F) {$\vline\tableau{ 3 & 2 \\ 2 & 1 & 1 \\ & \\\hline}$};
      \node at (4,3)  (H) {$\vline\tableau{ 3 & 3 \\ 2 & 1 & 1 \\ & \\\hline}$};
      \node at (5,3)  (I) {$\vline\tableau{ 3 & 1 \\ 2 & 2 & 2 \\ & \\\hline}$};
      \node at (5,2)  (L) {$\vline\tableau{ 3 & 3 \\ 2 & 2 & 1 \\ & \\\hline}$};
      \node at (6,1)  (N) {$\vline\tableau{ 3 & 3 \\ 2 & 2 & 2 \\ & \\\hline}$};
      \draw[thick,blue  ,->](A.south) -- (C)   node[midway,above]{$1$};
      \draw[thick,blue  ,->](B.south) -- (F)   node[midway,above]{$1$};
      \draw[thick,blue  ,->](D.south) -- (H)   node[midway,above]{$1$};
      \draw[thick,blue  ,->](F.south) -- (I)   node[midway,above]{$1$};
      \draw[thick,blue  ,->](H.south) -- (L)   node[midway,above]{$1$};
      \draw[thick,blue  ,->](L.south) -- (N)   node[midway,above]{$1$};
      \draw[thick,purple,->](A) -- (B)   node[midway,left ]{$2$};
      \draw[thick,purple,->](B) -- (D)   node[midway,left ]{$2$};
      \draw[thick,purple,->](F) -- (H)   node[midway,left ]{$2$};
    \end{tikzpicture}
    \caption{\label{fig:demazure-crystal-key}The Demazure crystals $\B_{123}(3,2,0)$ (left), $\B_{132}(3,2,0)$ (middle) and $\B_{312}(3,2,0)$ (right) on semistandard key tableaux with edges $\f_1 {\color{blue}\searrow}$, $\f_2 {\color{purple}\downarrow}$.}
  \end{center}
\end{figure}
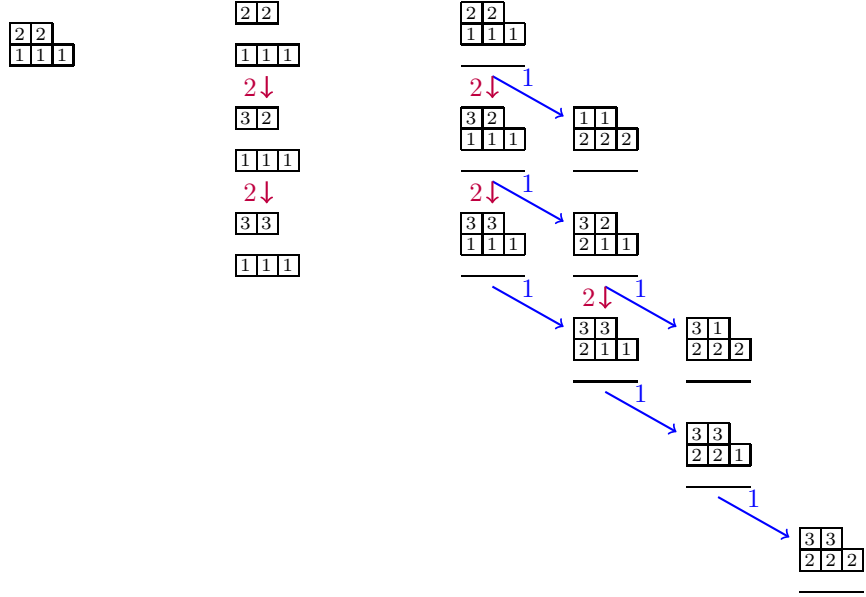

This implicitly defines the lowering operators $\f_i$. One advantage to using the semistandard key tableaux paradigm for Demazure crystals is that the Demazure truncation, instances where $\f_i(T)=0$ in the Demazure case but not in the full crystal, becomes evident from condition (3) of Definition~\ref{def:SSKT}.

\begin{theorem}[\cite{ASc18}]
  The data $(\SSKT(\comp{a}),\wt,\{\e_i,\f_i\}_{1 \leq i < n})$ determines the Demazure crystal $\B_w(\lambda)$ where $\comp{a} = w \cdot \lambda$.
  \label{thm:crystal-dem}
\end{theorem}

\subsection{Diagram crystals}
\label{sec:crystal-koh}

Given a weak composition $\comp{a}$ and the partition $\lambda$ to which it sorts, Assaf and Searles \cite[Definition~4.5]{AS18} give an injective map from Kohnert diagrams of $\comp{a}$ to semistandard Young tableaux of shape $\lambda$. 

\begin{definition}[\cite{AS18}]
  Given a weak composition $\comp{a}$ of length $n$ that sorts to the partition $\lambda$, define $\varphi: \KD(\D(\comp{a})) \rightarrow \SSYT_n(\lambda)$ by replacing each cell in row $i$ with entry $n-i+1$ and sorting the columns to increase bottom to top.
  \label{def:embed-SSYT}
\end{definition}

Inspired by this, we construct Demazure crystal operators on diagrams using the following pairing rule.

\begin{definition}
  For $T$ a diagram and $i \geq 1$ an integer, the \newword{$i$-pairing} of cells of $T$ in rows $i$ and $i+1$ is defined as follows:
  \begin{itemize}
  \item $i$-pair cells in rows $i$ and $i+1$ whenever they appear in the same column,
  \item iteratively $i$-pair an unpaired cell in row $i$ with an unpaired cell in row $i+1$ to its right whenever all cells in rows $i$ and $i+1$ that lie strictly between them are already $i$-paired.
  \end{itemize}
  \label{def:KD-pair}
\end{definition}

Given the asymmetry between raising and lowering operators for Demazure crystals, it is more natural to define the raising operators on diagrams directly and the lowering operators implicitly.

\begin{definition}
  For $T$ a diagram and $i \geq 1$ an integer, the \newword{raising operator} $\Ke_i$ acts on $T$ as follows: if $T$ has no unpaired cell in row $i+1$, then $\Ke_i(T)=0$; else, move the rightmost unpaired cell in row $i+1$ down to row $i$, staying within its column, leaving all other cells unmoved.
  \label{def:KD-raise}
\end{definition}

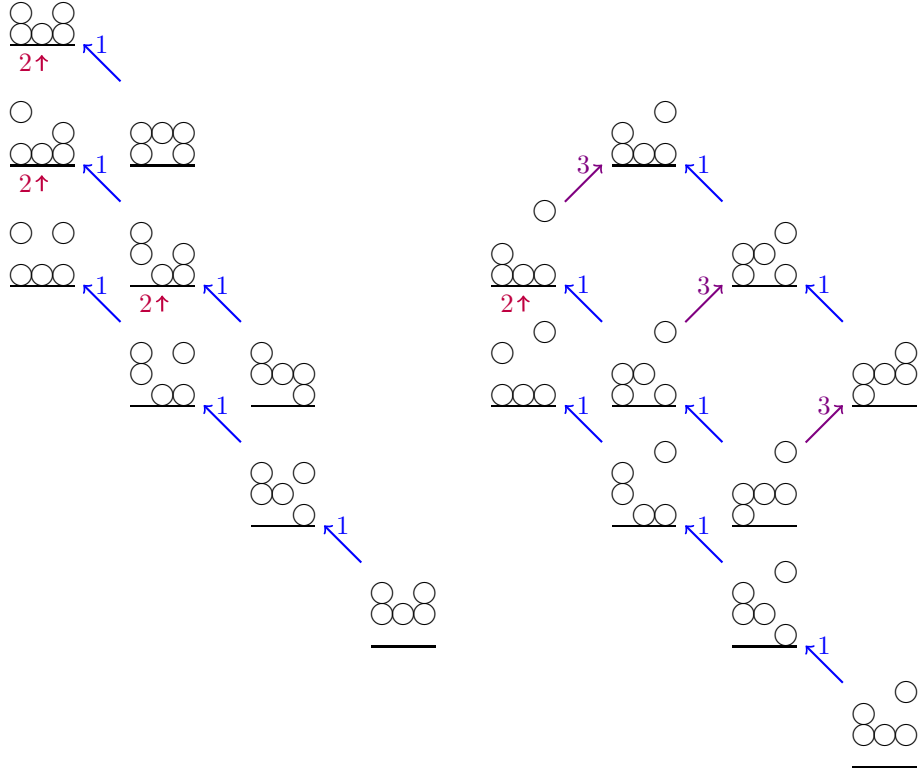
\begin{figure}[ht]
  \begin{center}
    \begin{tikzpicture}[xscale=1.6,yscale=1.6]
      \node at (9,0)  (A) {$\cirtab{ & & \ \\ \ & & \\ \ & \ & \ \\ & & \\\hline}$};
      \node at (5,1)  (B) {$\cirtab{ & & \\ \ & & \ \\ \ & \ & \ \\ & & \\\hline}$};
      \node at (8,2)  (C) {$\cirtab{ & & \ \\ & & \\ \ & \ & \ \\ \ & & \\\hline}$};
      \node at (8,1)  (D) {$\cirtab{ & & \ \\ \ & & \\ \ & \ & \\ & & \ \\\hline}$};
      \node at (4,2)  (E) {$\cirtab{ & & \\ \ & & \ \\ \ & \ & \\ & & \ \\\hline}$};
      \node at (9,3)  (F) {$\cirtab{ & & \\ & & \ \\ \ & \ & \ \\ \ & & \\\hline}$};
      \node at (7,3)  (G) {$\cirtab{ & & \ \\ & & \\ \ & \ & \\ \ & & \ \\\hline}$};
      \node at (7,2)  (H) {$\cirtab{ & & \ \\ \ & & \\ \ & & \\ & \ & \ \\\hline}$};
      \node at (3,3)  (I) {$\cirtab{ & & \\ \ & & \ \\ \ & & \\ & \ & \ \\\hline}$};
      \node at (4,3)  (J) {$\cirtab{ & & \\ \ & & \\ \ & \ & \ \\ & & \ \\\hline}$};
      \node at (8,4)  (K) {$\cirtab{ & & \\ & & \ \\ \ & \ & \\ \ & & \ \\\hline}$};
      \node at (6,3)  (L) {$\cirtab{ & & \ \\ \ & & \\ & & \\ \ & \ & \ \\\hline}$};
      \node at (2,4)  (M) {$\cirtab{ & & \\ \ & & \ \\ & & \\ \ & \ & \ \\\hline}$};
      \node at (3,4)  (N) {$\cirtab{ & & \\ \ & & \\ \ & & \ \\ & \ & \ \\\hline}$};
      \node at (3,5)  (O) {$\cirtab{ & & \\ & & \\ \ & \ & \ \\ \ & & \ \\\hline}$};
      \node at (6,4)  (P) {$\cirtab{ & & \ \\ & & \\ \ & & \\ \ & \ & \ \\\hline}$};
      \node at (2,5)  (Q) {$\cirtab{ & & \\ \ & & \\ & & \ \\ \ & \ & \ \\\hline}$};
      \node at (7,5)  (R) {$\cirtab{ & & \\ & & \ \\ \ & & \\ \ & \ & \ \\\hline}$};
      \node at (2,6)  (S) {$\cirtab{ & & \\ & & \\ \ & & \ \\ \ & \ & \ \\\hline}$};
      \draw[thick,blue  ,->](A) -- (D)   node[midway,above]{$1$};
      \draw[thick,blue  ,->](D) -- (H)   node[midway,above]{$1$};
      \draw[thick,blue  ,->](H) -- (L)   node[midway,above]{$1$};
      \draw[thick,blue  ,->](C) -- (G)   node[midway,above]{$1$};
      \draw[thick,blue  ,->](G) -- (P)   node[midway,above]{$1$};
      \draw[thick,blue  ,->](F) -- (K)   node[midway,above]{$1$};
      \draw[thick,blue  ,->](K) -- (R)   node[midway,above]{$1$};
      \draw[thick,blue  ,->](B) -- (E)   node[midway,above]{$1$};
      \draw[thick,blue  ,->](E) -- (I)   node[midway,above]{$1$};
      \draw[thick,blue  ,->](I) -- (M)   node[midway,above]{$1$};
      \draw[thick,blue  ,->](J) -- (N)   node[midway,above]{$1$};
      \draw[thick,blue  ,->](N) -- (Q)   node[midway,above]{$1$};
      \draw[thick,blue  ,->](O) -- (S)   node[midway,above]{$1$};
      \draw[thick,purple,->](L) -- (P)   node[midway,left ]{$2$};
      \draw[thick,purple,->](I) -- (N)   node[midway,left ]{$2$};
      \draw[thick,purple,->](M) -- (Q)   node[midway,left ]{$2$};
      \draw[thick,purple,->](Q) -- (S)   node[midway,left ]{$2$};
      \draw[thick,violet,->](C) -- (F)   node[midway,above]{$3$};
      \draw[thick,violet,->](G) -- (K)   node[midway,above]{$3$};
      \draw[thick,violet,->](P) -- (R)   node[midway,above]{$3$};
    \end{tikzpicture}
    \caption{\label{fig:kohnert-crystal}The Kohnert crystal on the Kohnert diagrams for the topmost diagram, with crystal edges $\Ke_1 {\color{blue}\nwarrow}$, $\Ke_2 {\color{purple}\uparrow}$, $\Ke_3 {\color{violet}\nearrow}$.}
  \end{center}
\end{figure}

As initial motivation for this construction, we have the following observation.

\begin{proposition}
  For a weak composition $\comp{a}$, the map $\psi$ sending $T\in\SSKT(\comp{a})$ to the diagram with a cell in position $(c,r)$ of $\psi(T)$ if and only if there is an entry $r$ in column $c$ of $T$ is a weight-preserving bijection $\SSKT(\comp{a}) \stackrel{\sim}{\rightarrow} \KD(\D(\comp{a}))$ satisfying $\psi(\e_i(T)) = \Ke_i(\psi(T))$ for all $i \geq 0$.
\end{proposition}

\begin{proof}
  The map $\psi$ corresponds to the map in \cite[Definition~3.14]{Ass-W} that is proved to be a weight-preserving bijection in \cite[Theorem~3.15]{Ass-W}. To see the intertwining of the crystal operators, notice first that the pairing rules correspond exactly, and the column of the entry that changes in Definition~\ref{def:SSKT-raise} agrees with the column of the cell that moves in Definition~\ref{def:KD-raise}. The result now follows by noticing the set of entries within every other column remains constant in Definition~\ref{def:SSKT-raise}.
\end{proof}

To justify our construction, we will prove the Kohnert crystal on $\KD(D)$ for southwest diagrams $D$ is a Demazure crystal. Taking the character, this resolves Conjecture~\ref{conj:demazure}.

%
\section{Kohnert crystals}
%
\label{sec:proof}

We embed Kohnert crystals into tableaux crystals to prove each connected component of the Kohnert crystal is a subset of a highest weight crystal.

\subsection{Closure}
\label{sec:closed}

From the pairing rule, the proposed raising operators $\Ke_i$ are well-defined on all diagrams. However, crystal operators are maps $\e_i,\f_i:\B \rightarrow \B\cup\{0\}$. Thus we must show that these operators do not leave the set of Kohnert diagrams in which we begin, a fact that holds whenever the initial diagram is southwest.

Comparing the poset of Kohnert moves (Fig.~\ref{fig:kohnert-poset}) with the crystal structure on Kohnert diagrams (Fig.~\ref{fig:kohnert-crystal}), it is clear that, while the raising operators $\Ke_i$ lower cells from row $i+1$ down to row $i$, these are not, in general, Kohnert moves. Nevertheless, when we begin with a southwest diagram $D$, each application of the raising operator on $T\in\KD(D)$ can be reconstructed from a sequence of Kohnert moves and reverse Kohnert moves \emph{within} $\KD(D)$.

\begin{theorem}
  For $D$ a southwest diagram and $T \in \KD(D)$, if $\Ke_r(T) \neq 0$ for some positive row index $r$, then $\Ke_r(T) \in \KD(D)$.
  \label{thm:closed}
\end{theorem}

\begin{proof}
  Suppose $\Ke_r$ acts on $T$ by lowering the cell $x$ in column $c_0$ from row $r+1$ down to row $r$. If $x$ is the rightmost cell in row $r+1$, then this is a Kohnert move, and so $\Ke_r(T) \in \KD(T) \subseteq \KD(D)$. Otherwise, let $z$ be the leftmost cell in row $r+1$, say in column $c_2$, that lies strictly right of column $c_0$, so that $c_0 < c_2$. In order for $\Ke_r$ to act by moving $x$, there must be no cell in column $c_0$, row $r$. Let $y$ be the leftmost cell in row $r$, say in column $c_1$, that lies strictly right of column $c_0$, so that $c_0 < c_1$. Then we must have $c_1 \leq c_2$ as well, else $\Ke_i$ would act by moving $z$ instead of $x$. See Fig.~\ref{fig:raise-kohnert} for an illustration.

  \begin{figure}[ht]
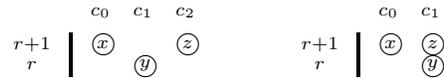

    \begin{displaymath}
      \begin{array}{rl}
        & \nulltab{ & c_0 & & c_1 & & c_2 } \\
        \nulltab{r+1 \\ r} & \vline\cirtab{ & x & & & & z \\ & & & y }
      \end{array}
      \hspace{4\cellsize}
      \begin{array}{rl}
        & \nulltab{ & c_0 & & c_1 } \\
        \nulltab{r+1 \\ r} & \vline\cirtab{ & x & & z \\ & & & y }
      \end{array}
    \end{displaymath}
    \caption{\label{fig:raise-kohnert}An illustration of the two possible scenarios when the raising operator $\Ke_r$ does not act by a Kohnert move.}
  \end{figure}
  
  Consider the first scenario in Fig.~\ref{fig:raise-kohnert}. Here, $T$ is not southwest as evidenced by the cells $x,y$. Since $T \in \KD(D)$ for a southwest diagram $D$, this means that $T \neq D$, and so we must be able to perform reverse Kohnert moves on $T$. In particular, there must be some sequence of reverse Kohnert moves that results in $y$ being lifted to the same row as (or above) $x$, since the position of $y$ precludes moving a cell in column $c_0$ up to row $r$ by reverse Kohnert moves. Then, since raising $y$ is not currently a reverse Kohnert move due to the position of $z$, there must exist some sequence of reverse Kohnert moves that results in $z$ (and, by iterating the argument, all cells to the right of this in row $r+1$) being lifted to a row strictly above $r+1$. Let $S$ denote the resulting Kohnert diagram of $D$. Then we apply a Kohnert move to $x$, still in row $r+1$, column $c_0$ in $S$, followed by Kohnert moves that undo the reverse Kohnert moves lifting $z$ to above row $r+1$. The result is $\Ke_i(T)$, obtained by a sequence of Kohnert moves on $S$, and so $\Ke_i(T) \in \KD(S)\subseteq \KD(D)$.
  
  Now consider the second scenario in Fig.~\ref{fig:raise-kohnert}. Again, $T$ is not southwest due to the cells $x,y$. As before, there must be some sequence of reverse Kohnert moves that results in the cell $y$ being lifted to row $r+1$ (or above). Equivalently, since cells are indistinguishable, there is a sequence of reverse Kohnert moves that lifts the cell $z$ above row $r+1$. Inductively, all cells in row $r+1$ to the right of column $c_0$ may be lifted above row $r+1$ by a sequence of reverse Kohnert moves, after which we may use a Kohnert move to push $x$ down to row $r$, then a series of Kohnert moves undoing the previous reverse Kohnert moves. The result is $\Ke_i(T)$, obtained by a sequence of Kohnert moves on a diagram in $\KD(D)$.
\end{proof}

Conversely to Theorem~\ref{thm:closed}, if $D$ is not southwest, then it can happen that $\Ke_i(T) \not\in \KD(D)$ for some $T\in\KD(D)$ for which $\Ke_i(T)\neq 0$. Thus we restrict our attention henceforth to southwest diagrams.

\begin{definition}
  Given a southwest diagram $D$, the \newword{Kohnert crystal on $\KD(D)$} consists of the following data: the set $\KD(D)$, the weight map $\wt:\KD(D)\rightarrow\mathbb{N}^n$, crystal raising operators $\Ke_i$, and crystal lowering operators $\Kf_i$ defined by $\Kf_i(T) = T'$ whenever $\Ke_i(T')=T$ and $\Kf_i(T) = 0$ otherwise. 
  \label{def:kohnert-crystal}
\end{definition}

For example, Fig.~\ref{fig:kohnert-crystal} shows the Kohnert crystal for the topmost diagram. Notice there are two connected components, one of which is isomorphic as a directed, colored graph to $\B_{312}(3,2,0)$ from Fig.~\ref{fig:demazure-crystal}.

The as yet undiscussed lowering operators must be inverse to the well-defined raising operators when both are nonzero. The following reveals their definition.

\begin{lemma}
  Let $T$ be a diagram and $i\geq 1$ an integer such that $\Ke_i(T) \neq 0$. Then, for $r=i,i+1$, a cell $(c,r) \in T$ is $i$-paired in $T$ if and only if $(c,r) \in \Ke_i(T)$ and is $i$-paired in $\Ke_i(T)$. Moreover, if $\Ke_i$ acts on $T$ by pushing the cell in position $(c,i+1)$ down to position $(c,i)$, then $(c,i)$ is the leftmost cell in row $i$ of $\Ke_i(T)$ that is not $i$-paired.
  \label{lem:KD-pairing}
\end{lemma}

\begin{proof}
  Suppose $\Ke_i$ acts on $T$ by pushing the cell in position $(c,i+1)$ down to position $(c,i)$. By Definition~\ref{def:KD-raise}, $(c,i+1)$ is the rightmost cell in row $i+1$ that is unpaired, and by the $i$-pairing algorithm, every cell in row $i+1$ to the right of column $c$ is $i$-paired with a cell in row $i$ that is also to the right of column $c$. Thus there is no cell in row $i+1$ to which the cell in position $(c,i)$ of $\Ke_i(T)$ can $i$-pair, ensuring it is unpaired in $\Ke_i(T)$. The other cells not in column $c$ remain as they were, and so the same set of cells in rows $i$ and $i+1$ not in column $c$ are $i$-paired in $T$ and $\Ke_i(T)$ the same way. Thus if some cell in row $i$, say in column $c'<c$, is not $i$-paired in $\Ke_i(T)$, then it is also not $i$-paired in $T$, and so, by Definition~\ref{def:KD-raise}, it would be $i$-paired with $(c,i+1)$ in $T$, a contradiction. Therefore the $(c,i)$ is indeed the leftmost cell in row $i$ of $\Ke_i(T)$ that is not $i$-paired.
\end{proof}

In contrast with the raising operators, the lowering operator applied to $T\in\KD(D)$ might result in a diagram not in $\KD(D)$, even for $D$ southwest.

\begin{definition}
  Given a diagram $D$, for $T\in\KD(D)$ and $i \geq 1$ an integer, the \newword{lowering operator} $\Kf_i$ acts on $T$ as follows. If $T$ has an unpaired cell in row $i$, then let $T'$ be the result of moving the leftmost unpaired cell in row $i$ up to row $i+1$, staying within its column, leaving all other cells unmoved. If $T$ has no unpaired cell in row $i$ or if $T'\not\in\KD(D)$, then set $\Kf_i(T)=0$; else set $\Kf_i(T)=T'$.
  \label{def:KD-lower}
\end{definition}

By Lemma~\ref{lem:KD-pairing}, the raising and lowering operators are inverse when nonzero.

\begin{proposition}
  Let $D$ be a southwest diagram. Then for $S,T\in\KD(D)$, we have $\Ke_i(T) = S$ if and only if $\Kf_i(S)=T$.
  \label{prop:inverse}
\end{proposition}

\begin{proof}
  Suppose $\Ke_i(T)=S$ for some $T\in\KD(D)$. Then $S\in\KD(D)$ and, by Lemma~\ref{lem:KD-pairing}, the only difference between $T$ and $S$ is that the rightmost unpaired cell in row $i+1$ of $T$ moves down to become the leftmost unpaired cell in row $i$ of $S$. Therefore $\Kf_i(S)=T$. Conversely, if $\Kf_i(S)=T$ for some $S\in\KD(D)$, then by Definition~\ref{def:KD-lower}, we have $T \in \KD(D)$. The argument in Lemma~\ref{lem:KD-pairing} is easily reversed to see that the difference between $S$ and $T$ is again characterized in the same way, and so $\Ke_i(T)=S$.
\end{proof}

Notice in Definition~\ref{def:KD-lower} $\Kf_i$ depends on the diagram $D$, whereas Definition~\ref{def:KD-raise} is independent of $D$.

\subsection{Rectification}
\label{sec:rectify}

To see the raising operators $\Ke_i$ have the basic structure of crystal operators, we define an injective map from $\KD(D)$ to a disjoint union of tableaux crystals that intertwines the crystal operators. Taking advantage of the map in Definition~\ref{def:embed-SSYT} from the Kohnert diagrams of a composition diagram to tableaux, it is enough to map each diagram in $\KD(D)$ to a diagram in $\KD(\D(\comp{a}))$ for some weak composition $\comp{a}$. For this, we use a characterization of diagrams obtainable by Kohnert moves on composition diagrams given in \cite[Lemma~2.2]{AS18}.

\begin{lemma}[\cite{AS18}]
  A diagram $T$ can be obtained via a series of Kohnert moves on a composition diagram if and only if for column index $c\geq 1$ and every row index $r\geq 1$, we have
  \begin{equation}
    \#\{ (c,s) \in T \mid s \geq r \} \geq \#\{ (c+1,s) \in T \mid s \geq r \}.
    \label{e:2.2}
  \end{equation}
  \label{lem:2.2}
\end{lemma}

We can think of Eq.~\eqref{e:2.2} as a diagram analog of the weakly decreasing condition on weak compositions in the sense that we will define crystal-like operators that take a given diagram to a canonical diagram satisfying Eq.~\eqref{e:2.2}. Given this, it is helpful to introduce terminology for the diagrams characterized by Eq.~\eqref{e:2.2}.

\begin{definition}
  A diagram $T$ is \newword{rectified} if it satisfies Eq.~\eqref{e:2.2} or, equivalently, if $T \in \KD(\D(\comp{a}))$ for some weak composition $\comp{a}$.
  \label{def:rectified}
\end{definition}

\begin{figure}[ht]
  \begin{center}
    \begin{tikzpicture}[xscale=0.3,yscale=0.3,
      every node/.style={inner sep=-0.4\cellsize}]
      \node at (3,5) (A35) {$\cball{red}{}$};
      \node at (4,5) (A45) {$\circify{\ }$};
      \node at (1,4) (A14) {$\circify{\ }$};
      \node at (3,4) (A34) {$\cball{red}{}$};
      \node at (4,3) (A43) {$\circify{\ }$};
      \node at (5,3) (A53) {$\circify{\ }$};
      \node at (7,3) (A73) {$\cball{red}{}$};
      \node at (1,2) (A12) {$\circify{\ }$};
      \node at (2,2) (A22) {$\circify{\ }$};
      \node at (3,2) (A32) {$\circify{\ }$};
      \node at (5,2) (A52) {$\circify{\ }$};
      \node at (6,2) (A62) {$\circify{\ }$};
      \node at (4,1) (A41) {$\circify{\ }$};
      \node at (8,1) (A81) {$\circify{\ }$};
      \draw[black] (0.5,5.5) -- (0.5,0.5) -- (8.5,0.5) ;
      \draw[violet] (A35) -- (A45) ;
      \draw[violet] (A34) -- (A43) ;
      \draw[violet] (A43) -- (A53) ;
      \draw[violet] (A12) -- (A22) ;
      \draw[violet] (A22) -- (A32) ;
      \draw[violet] (A45) -- (A52) ;
      \draw[violet] (A52) -- (A62) ;
      \draw[violet] (A32) -- (A41) ;
      \draw[violet] (A73) -- (A81) ;
      \node at (13,5) (B35) {$\circify{\ }$};
      \node at (14,5) (B45) {$\circify{\ }$};
      \node at (13,4) (B34) {$\circify{\ }$};
      \node at (14,4) (B44) {$\circify{\ }$};
      \node at (15,3) (B53) {$\circify{\ }$};
      \node at (16,3) (B63) {$\circify{\ }$};
      \node at (17,3) (B73) {$\circify{\ }$};
      \node at (13,2) (B32) {$\circify{\ }$};
      \node at (14,2) (B42) {$\circify{\ }$};
      \node at (15,2) (B52) {$\circify{\ }$};
      \node at (16,2) (B62) {$\circify{\ }$};
      \node at (17,2) (B72) {$\circify{\ }$};
      \node at (15,1) (B51) {$\circify{\ }$};
      \node at (18,1) (B81) {$\circify{\ }$};
      \draw[thin,black] (12.5,5.5) -- (12.5,0.5) -- (18.5,0.5) ;
      \draw[thin,violet] (B35) -- (B45) ;
      \draw[thin,violet] (B34) -- (B44) ;
      \draw[thin,violet] (B44) -- (B53) ;
      \draw[thin,violet] (B53) -- (B63) ;
      \draw[thin,violet] (B63) -- (B73) ;
      \draw[thin,violet] (B32) -- (B42) ;
      \draw[thin,violet] (B42) -- (B52) ;
      \draw[thin,violet] (B52) -- (B62) ;
      \draw[thin,violet] (B62) -- (B72) ;
      \draw[thin,violet] (B45) -- (B51) ;
      \draw[thin,violet] (B72) -- (B81) ;
    \end{tikzpicture}
    \caption{\label{fig:lemma2.2}The column pairing rule on an unrectified (left) and a rectified (right) diagram, where cells right of the first column not column paired to the left are highlighted.}
  \end{center}
\end{figure}
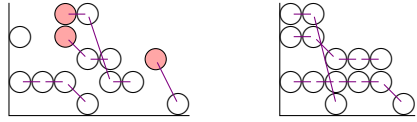

For example, in Fig.~\ref{fig:lemma2.2}, the diagram on the left is rectified whereas the diagram on the right is not. To define our rectification operators, we begin with a column pairing rule that is precisely the transpose of the row $i$-pairing rule in Definition~\ref{def:KD-pair}.

\begin{definition}
  For $T$ a diagram and $i \geq 1$ an integer, the \newword{column $i$-pairing} of cells of $T$ in columns $i$ and $i+1$ is defined as follows:
  \begin{itemize}
  \item $i$-pair cells in columns $i$ and $i+1$ whenever they appear in the same row,
  \item iteratively $i$-pair an unpaired cell in column $i+1$ with an unpaired cell in column $i$ above it whenever all cells in columns $i$ and $i+1$ that lie strictly between them are already $i$-paired.
  \end{itemize}
  \label{def:KD-cpair}
\end{definition}

Given the upward direction for column $i$-pairing, there will be an unpaired cell in column $i+1$ if and only if Eq.~\eqref{e:2.2} fails. For example, Fig.~\ref{fig:lemma2.2} shows the column pairings for two diagram, one rectified and one not.

\begin{figure}[ht]
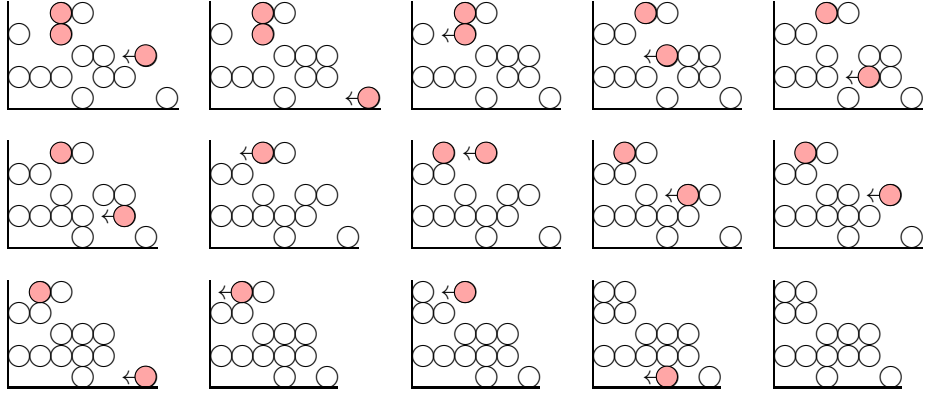

  \begin{displaymath}
    \arraycolsep=0.75\cellsize
    \begin{array}{lllll}
    \cirtab{%
      & & \cball{red}{} & \ \\
      \ & & \cball{red}{} \\
      & & & \ & \ & & \leftball{red}{} \\
      \ & \ & \ & & \ & \ \\
      & & & \ & & & & \ \\\hline } &
    \cirtab{%
      & & \cball{red}{} & \ \\
      \ & & \cball{red}{} \\
      & & & \ & \ & \ & \\
      \ & \ & \ & & \ & \ \\
      & & & \ & & & & \leftball{red}{} \\\hline } &
    \cirtab{%
      & & \cball{red}{} & \ \\
      \ & & \leftball{red}{} \\
      & & & \ & \ & \ & \\
      \ & \ & \ & & \ & \ \\
      & & & \ & & & \ \\\hline } &
    \cirtab{%
      & & \cball{red}{} & \ \\
      \ & \ \\
      & & & \leftball{red}{} & \ & \ & \\
      \ & \ & \ & & \ & \ \\
      & & & \ & & & \ \\\hline } &
    \cirtab{%
      & & \cball{red}{} & \ \\
      \ & \ \\
      & & \ & & \ & \ & \\
      \ & \ & \ & & \leftball{red}{} & \ \\
      & & & \ & & & \ \\\hline } \\ \\
    \cirtab{%
      & & \cball{red}{} & \ \\
      \ & \ \\
      & & \ & & \ & \ & \\
      \ & \ & \ & \ &  & \leftball{red}{} \\
      & & & \ & & & \ \\\hline } &
    \cirtab{%
      & & \leftball{red}{} & \ \\
      \ & \ \\
      & & \ & & \ & \ & \\
      \ & \ & \ & \ & \ \\
      & & & \ & & & \ \\\hline } &
    \cirtab{%
      & \cball{red}{} & & \leftball{red}{} \\
      \ & \ \\
      & & \ & & \ & \ & \\
      \ & \ & \ & \ & \ \\
      & & & \ & & & \ \\\hline } &
    \cirtab{%
      & \cball{red}{} & \ \\
      \ & \ \\
      & & \ & & \leftball{red}{} & \ & \\
      \ & \ & \ & \ & \ \\
      & & & \ & & & \ \\\hline } &
    \cirtab{%
      & \cball{red}{} & \ \\
      \ & \ \\
      & & \ & \ & & \leftball{red}{} & \\
      \ & \ & \ & \ & \ \\
      & & & \ & & & \ \\\hline } \\ \\
    \cirtab{%
      & \cball{red}{} & \ \\
      \ & \ \\
      & & \ & \ & \ & \\
      \ & \ & \ & \ & \ \\
      & & & \ & & & \leftball{red}{} \\\hline } &
    \cirtab{%
      & \leftball{red}{} & \ \\
      \ & \ \\
      & & \ & \ & \ & \\
      \ & \ & \ & \ & \ \\
      & & & \ & & \ \\\hline } &
    \cirtab{%
      \ & & \leftball{red}{} \\
      \ & \ \\
      & & \ & \ & \ & \\
      \ & \ & \ & \ & \ \\
      & & & \ & & \ \\\hline } &
    \cirtab{%
      \ & \ \\
      \ & \ \\
      & & \ & \ & \ & \\
      \ & \ & \ & \ & \ \\
      & & & \leftball{red}{} & & \ \\\hline } &
    \cirtab{%
      \ & \ \\
      \ & \ \\
      & & \ & \ & \ & \\
      \ & \ & \ & \ & \ \\
      & & \ & & & \ \\\hline } 
    \end{array}
  \end{displaymath}  
  \caption{\label{fig:rectification}The rectification of a diagram, where the cells right of the first column not column paired to the left are highlighted and the cell to rectify at each step is indicated with arrows.}
\end{figure}

\begin{definition}
  For $T$ a diagram and $i \geq 1$ an integer, the \newword{rectification operator} $\Rect_i$ acts on $T$ as follows: if $T$ has no unpaired cell in column $i+1$, then $\Rect_i(T)=T$; else, move the lowest unpaired cell in column $i+1$ left to column $i$, staying within its row, leaving all other cells unmoved.
  \label{def:KD-rect}
\end{definition}

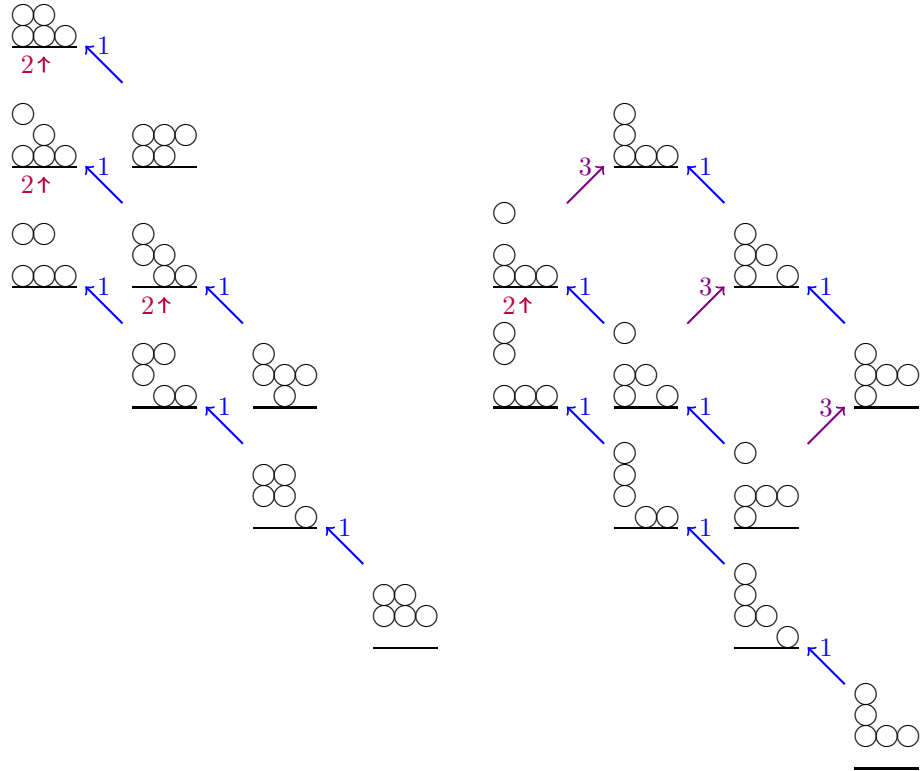
\begin{figure}[ht]
  \begin{center}
    \begin{tikzpicture}[xscale=1.6,yscale=1.6]
      \node at (9,0)  (A) {$\cirtab{ \ \\ \ & & \\ \ & \ & \ \\ & & \\\hline}$};
      \node at (5,1)  (B) {$\cirtab{ & & \\ \ & \ \\ \ & \ & \ \\ & & \\\hline}$};
      \node at (8,2)  (C) {$\cirtab{ \ \\ & & \\ \ & \ & \ \\ \ & & \\\hline}$};
      \node at (8,1)  (D) {$\cirtab{ \ \\ \ & & \\ \ & \ & \\ & & \ \\\hline}$};
      \node at (4,2)  (E) {$\cirtab{ & & \\ \ & \ \\ \ & \ & \\ & & \ \\\hline}$};
      \node at (9,3)  (F) {$\cirtab{ & & \\ \ \\ \ & \ & \ \\ \ & & \\\hline}$};
      \node at (7,3)  (G) {$\cirtab{ \ \\ & & \\ \ & \ & \\ \ & & \ \\\hline}$};
      \node at (7,2)  (H) {$\cirtab{ \ \\ \ & & \\ \ & & \\ & \ & \ \\\hline}$};
      \node at (3,3)  (I) {$\cirtab{ & & \\ \ & \ \\ \ & & \\ & \ & \ \\\hline}$};
      \node at (4,3)  (J) {$\cirtab{ & & \\ \ & & \\ \ & \ & \ \\ & \ \\\hline}$};
      \node at (8,4)  (K) {$\cirtab{ & & \\ \ \\ \ & \ & \\ \ & & \ \\\hline}$};
      \node at (6,3)  (L) {$\cirtab{ \ \\ \ & & \\ & & \\ \ & \ & \ \\\hline}$};
      \node at (2,4)  (M) {$\cirtab{ & & \\ \ & \ \\ & & \\ \ & \ & \ \\\hline}$};
      \node at (3,4)  (N) {$\cirtab{ & & \\ \ & & \\ \ & \ \\ & \ & \ \\\hline}$};
      \node at (3,5)  (O) {$\cirtab{ & & \\ & & \\ \ & \ & \ \\ \ & \ \\\hline}$};
      \node at (6,4)  (P) {$\cirtab{ \ \\ & & \\ \ & & \\ \ & \ & \ \\\hline}$};
      \node at (2,5)  (Q) {$\cirtab{ & & \\ \ & & \\ & \ \\ \ & \ & \ \\\hline}$};
      \node at (7,5)  (R) {$\cirtab{ & & \\ \ \\ \ & & \\ \ & \ & \ \\\hline}$};
      \node at (2,6)  (S) {$\cirtab{ & & \\ & & \\ \ & \ \\ \ & \ & \ \\\hline}$};
      \draw[thick,blue  ,->](A) -- (D)   node[midway,above]{$1$};
      \draw[thick,blue  ,->](D) -- (H)   node[midway,above]{$1$};
      \draw[thick,blue  ,->](H) -- (L)   node[midway,above]{$1$};
      \draw[thick,blue  ,->](C) -- (G)   node[midway,above]{$1$};
      \draw[thick,blue  ,->](G) -- (P)   node[midway,above]{$1$};
      \draw[thick,blue  ,->](F) -- (K)   node[midway,above]{$1$};
      \draw[thick,blue  ,->](K) -- (R)   node[midway,above]{$1$};
      \draw[thick,blue  ,->](B) -- (E)   node[midway,above]{$1$};
      \draw[thick,blue  ,->](E) -- (I)   node[midway,above]{$1$};
      \draw[thick,blue  ,->](I) -- (M)   node[midway,above]{$1$};
      \draw[thick,blue  ,->](J) -- (N)   node[midway,above]{$1$};
      \draw[thick,blue  ,->](N) -- (Q)   node[midway,above]{$1$};
      \draw[thick,blue  ,->](O) -- (S)   node[midway,above]{$1$};
      \draw[thick,purple,->](L) -- (P)   node[midway,left ]{$2$};
      \draw[thick,purple,->](I) -- (N)   node[midway,left ]{$2$};
      \draw[thick,purple,->](M) -- (Q)   node[midway,left ]{$2$};
      \draw[thick,purple,->](Q) -- (S)   node[midway,left ]{$2$};
      \draw[thick,violet,->](C) -- (F)   node[midway,above]{$3$};
      \draw[thick,violet,->](G) -- (K)   node[midway,above]{$3$};
      \draw[thick,violet,->](P) -- (R)   node[midway,above]{$3$};
    \end{tikzpicture}
    \caption{\label{fig:kohnert-crystal-rectified}The Demazure crystals $\B_{312}(3,2,0)$ (left) and $\B_{4123}(3,1,1,0)$ (right) on the Kohnert diagrams for composition diagrams, with crystal edges $\Ke_1 {\color{blue}\nwarrow}$, $\Ke_2 {\color{purple}\uparrow}$, $\Ke_3 {\color{violet}\nearrow}$.}
  \end{center}
\end{figure}

For example, Fig.~\ref{fig:rectification} shows the steps in the rectification of a diagram, where we choose the rightmost cell to rectify at each step. This choice does not affect the rectified diagram, as we prove in Lemma~\ref{lem:rectification} below. Foreshadowing Theorem~\ref{thm:embed}, Fig.~\ref{fig:kohnert-crystal-rectified} shows the rectification of the Kohnert diagrams from Fig.~\ref{fig:kohnert-crystal}.

Comparing Definition~\ref{def:KD-rect} with Definition~\ref{def:KD-raise}, the rectification of $T$ is obtained by transposing $T$ along the line $y=x$, applying the raising operator (if nonzero), and transposing back. Perhaps, then, it comes as little surprise that rectification and raising operators commute, in the following sense.

\begin{theorem}
  Given a diagram $T$, a row index $r \geq 1$ and a column index $c \geq 1$, $\Ke_r(T) \neq 0$ if and only if $\Ke_r(\Rect_c(T)) \neq 0$, and, in this case, $\Ke_r(\Rect_c(T)) = \Rect_c(\Ke_r(T))$.
  \label{thm:commute}
\end{theorem}

\begin{proof}
  If $\Rect_c(T) = T$, then the first assertion is trivial. For the second, suppose $\Ke_r$ acts on $T$ by moving the cell in position $(i,r+1)$ down to position $(i,r)$. If $i \neq c,c+1$, then cells in columns $c,c+1$ remain unmoved, so $\Rect_c(\Ke_r(T)) = \Ke_r(T) = \Ke_r(\Rect_c(T))$ as desired. If $i=c+1$, then the left hand side of Eq.~\eqref{e:2.2} is the same for $T$ as for $\Ke_r(T)$ while the right hand side is one smaller for $r+1$ and the same for other rows. In particular, since Eq.~\eqref{e:2.2} holds for $T$ at column $c$, it also holds for $\Ke_r(T)$ at column $c$, and so again $\Rect_c(\Ke_r(T)) = \Ke_r(T) = \Ke_r(\Rect_c(T))$. The interesting case occurs at $i=c$. If there is no cell in position $(c+1,r+1)$ in $T$, then the column $c$-pairing for $T$ and $\Ke_r(T)$ will be the same, once again giving $\Rect_c(\Ke_r(T)) = \Ke_r(T) = \Ke_r(\Rect_c(T))$. Otherwise, since $\Ke_r$ acts on the cell in position $(c,r+1)$, it must not be $r$-paired in $T$ and so there is no cell in position $(c,r)$. Moreover, the cell in position $(c+1,r+1)$ must be $r$-paired, and it must be with the cell in position $(c+1,r)$. Now, since Eq.~\eqref{e:2.2} holds for $T$ in column $c$ \emph{for all rows}, the inequality at row $r$ forces a \emph{strict} inequality at row $r+1$. In passing to $\Ke_r(T)$, the right hand side of Eq.~\eqref{e:2.2} for column $c$ is unchanged and the left hand side decreases by $1$ at row $r+1$ and remains the same elsewhere. Therefore the inequality holds for $\Ke_r(T)$ in column $c$, once again ensuring $\Rect_c(\Ke_r(T)) = \Ke_r(T) = \Ke_r(\Rect_c(T))$.

  Now suppose $\Rect_c$ acts nontrivially on $T$, say moving the cell in position $(c+1,j)$ left to position $(c,j)$ for some row index $j$. If $\Ke_r(T) = 0$, then, using the relationship between the definitions for raising and rectification, transposing the above argument ensures $\Ke_r(\Rect_c(T)) = 0$ as well, proving the first statement. For the second, we have four main cases to consider based on the position of $j$ relative to $r+1,r$.

    \textsc{Case} ($j>r+1$): Let $x$ denote the cell of $T$ in position $(c+1,j)$ that moves right to position $(c,j)$ under $\Rect_c$. Suppose $\Ke_r$ acts on $T$ by moving the cell $z$ in position $(i,r+1)$ down to position $(i,r)$. We have four subcases depicted in Fig.~\ref{fig:commute-r++}.
  \begin{itemize}
  \item If $i \neq c,c+1$, then the movement of $x$ does not affect rows $r,r+1$ and the movement of $z$ does not affect columns $c,c+1$, so $\Rect_c(\Ke_r(T)) = \Ke_r(\Rect_c(T))$. 
  \item (Fig.~\ref{fig:commute-r++} left) If $i=c+1$, then since $j>r+1$, $z$ is below $x$ and so is column $c$-paired in $T$, say with $y$, and since $z$ is in column $c+1$ and moves down, it remains column $c$-paired with $y$ in $\Ke_r(T)$. Thus $\Rect_c$ acts on $x$ in both $T$ and $\Ke_r(T)$ without affecting $r$-pairings, so $\Rect_c(\Ke_r(T)) = \Ke_r(\Rect_c(T))$. 
  \item (Fig.~\ref{fig:commute-r++} middle) If $i=c$ and there is no cell right of $z$ in position $(c+1,r+1)$, then the column $c$-pairings are the same for $\Ke_r(T)$ as for $T$, and so $\Rect_c(\Ke_r(T)) = \Ke_r(\Rect_c(T))$. 
  \item (Fig.~\ref{fig:commute-r++} right) If $i=c$ and there is a cell $y$ right of $z$ in position $(c+1,r+1)$, then since $z$ is the rightmost cell not $r$-paired in $T$, $y$ must be $r$-paired with the cell $w$ below it in position $(c+1,r)$. Since $\Rect_c$ acts on $T$ at $x$, the cell $w$ must be column $c$-paired with some cell $v$ above $z$. Then $v$ and $y$ are column $c$-paired in $\Ke_r(T)$, ensuring $\Rect_c$ acts on $x$ in $\Ke_r(T)$ as well. Thus $\Rect_c(\Ke_r(T)) = \Ke_r(\Rect_c(T))$. 
  \end{itemize}

    \begin{figure}[ht]
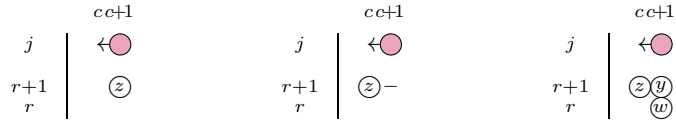

    \begin{displaymath}
      \begin{array}{rl}
        & \nulltab{ & c & c\!+\!1 } \\
        \nulltab{j \\ \\ r+1 \\ r} & \vline\nulltab{ & & \leftball{purple}{} \\ \\ & & \circify{z} \\ & & }
      \end{array}
      \hspace{5em}
      \begin{array}{rl}
        & \nulltab{ & c & c\!+\!1 } \\
        \nulltab{j \\ \\ r+1 \\ r} & \vline\nulltab{ & & \leftball{purple}{} \\ \\ & \circify{z} & - \\ & & }
      \end{array}
      \hspace{5em}
      \begin{array}{rl}
        & \nulltab{ & c & c\!+\!1 } \\
        \nulltab{j \\ \\ r+1 \\ r} & \vline\nulltab{ & & \leftball{purple}{} \\ \\ & \circify{z} & \circify{ y } \\ & & \circify{w} }
      \end{array}
    \end{displaymath}
    \caption{\label{fig:commute-r++}The subcases when $\Rect_c$ acts on $T$ by moving the cell $x$ in position $(c+1,j)$ left to position $(c,j)$, and $\Ke_r$ acts on $T$ by moving the cell $z$ in position $(i,r+1)$ down to position $(i,r)$.}
  \end{figure}
  
  \textsc{Case} ($j=r+1$): Let $x$ denote the cell of $T$ in position $(c+1,r+1)$ that moves left under $\Rect_c$. We have three subcases depicted in Fig.~\ref{fig:commute-r+}.
  \begin{itemize}
  \item (Fig.~\ref{fig:commute-r+} left) If there is no cell below $x$ in position $(c+1,r)$, but $x$ is  $r$-paired in $T$, then $x$ is $r$-paired to the same cell in $\Rect_c(T)$, making the action of $\Ke_r$ the same for both. Thus $\Rect_c(\Ke_r(T)) = \Ke_r(\Rect_c(T))$.
  \item (Fig.~\ref{fig:commute-r+} left) If there is no cell below $x$ in position $(c+1,r)$ and $x$ is not $r$-paired in $T$, then there is also no cell in position $(c,r)$. Therefore $x$ remains unpaired in $\Rect_c(T)$ as well, and so $\Ke_r$ acts on the same cell in both $T$ and $\Rect_c(T)$. Moreover, if $\Ke_r$ acts on $T$ by moving $x$ down, then since there is no cell in position $(c,r)$, $x$ remains the highest cell in $\Ke_r(T)$ with no column $c$-pairing, and so $\Rect_c$ acts on $\Ke_r(T)$ by pushing $x$ left. Either way, $\Rect_c(\Ke_r(T)) = \Ke_r(\Rect_c(T))$.
  \item (Fig.~\ref{fig:commute-r+} right) Finally, if there is a cell below $x$ in position $(c+1,r)$, then $x$ is $r$-paired with this cell. Moreover, the cell below $x$ must be column $c$-paired, and so there must be a cell $z$ in position $(c,r)$ as well. If $z$ is $r$-paired in $T$, say with some cell $y$ right of $x$, then the cell in position $(c+1,r)$ is $r$-paired with $y$ in $\Rect_c(T)$. Consequently, the $r$-pairings are the same for $T$ and $\Rect_c(T)$, and $\Ke_r$ acts in some column $i \neq c,c+1$, showing once again $\Rect_c(\Ke_r(T)) = \Ke_r(\Rect_c(T))$.
  \end{itemize}
    
    \begin{figure}[ht]
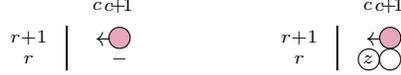

    \begin{displaymath}
      \begin{array}{rl}
        & \nulltab{ & c & c\!+\!1 } \\
        \nulltab{r+1 \\ r} & \vline\nulltab{ & & \leftball{purple}{} \\ & & - }
      \end{array}
      \hspace{5em}
      \begin{array}{rl}
        & \nulltab{ & c & c\!+\!1 } \\
        \nulltab{r+1 \\ r} & \vline\nulltab{ & & \leftball{purple}{} \\ & \circify{z} & \circify{ \ } }
      \end{array}
    \end{displaymath}
    \caption{\label{fig:commute-r+}The subcases when $\Rect_c$ acts on $T$ by moving the cell in position $(c+1,r+1)$ left to position $(c,r+1)$.}
  \end{figure}
  
  \textsc{Case} ($j=r$): Let $x$ denote the cell of $T$ in position $(c+1,r)$ that moves left under $\Rect_c$. We have four subcases depicted in Fig.~\ref{fig:commute-r}.
  \begin{itemize}
  \item (Fig.~\ref{fig:commute-r} left) If there is no cell above $x$ in position $(c+1,r+1)$, then either $x$ is not $r$-paired or it is $r$-paired with some cell strictly to its right. In either case, moving $x$ left to position $(c,r)$ does not change to which cell, if any, it is $r$-paired, and so the $r$-pairing is the same on $T$ as on $\Rect_c(T)$. Furthermore, since $x$ is not column $c$-paired, there is no cell in position $(c,r)$ nor in position $(c,r+1)$, since then $x$ would necessarily column pair with it. Therefore $\Ke_r$ acts on a column $i \neq c,c+1$, and so $\Rect_c(\Ke_r(T)) = \Ke_r(\Rect_c(T))$.
  \item (Fig.~\ref{fig:commute-r} middle) If there is a cell above $x$, then necessarily it is $r$-paired with $x$. If there is no cell in position $(c,r+1)$, then the cells in positions $(c,r)$ and $(c+1,r+1)$ are $r$-paired in $\Rect_c(T)$, ensuring that $\Ke_r$ acts in the same way in both since it cannot act in columns $c,c+1$. Therefore $\Rect_c(\Ke_r(T)) = \Ke_r(\Rect_c(T))$.
  \item (Fig.~\ref{fig:commute-r} right) Again with the cell above $x$ being $r$-paired with $x$, suppose there is a cell $z$ in position $(c,r+1)$. If $z$ is $r$-paired in $T$, say with some cell $y$, then $y$ is $r$-paired with the cell above $x$ in $\Rect_c(T)$, leaving all $r$-pairings unchanged and the action of $\Ke_r$ not in columns $c,c+1$, so $\Rect_c(\Ke_r(T)) = \Ke_r(\Rect_c(T))$.
  \item (Fig.~\ref{fig:commute-r} right) Finally, suppose the cell above $x$ exists, and the cell $z$ in position $(c,r+1)$ exists, and $z$ is not $r$-paired in $T$. If there is another cell right of $z$ not $r$-paired, then the action of $\Ke_r$ will be on that cell in both $T$ and $\Rect_c(T)$ giving $\Rect_c(\Ke_r(T)) = \Ke_r(\Rect_c(T))$. Otherwise, $\Ke_r$ acts on $T$ by moving $z$ down to position $(c,r)$, and then $\Rect_c$ acts on $\Ke_r(T)$ by moving the cell above $x$ left to position $(c,r+1)$. In the other direction, $\Ke_r$ acts on $\Rect_c(T)$ by moving the cell above $x$ down to position $(c+1,r)$. The end results agree, and so $\Rect_c(\Ke_r(T)) = \Ke_r(\Rect_c(T))$.
  \end{itemize}

    \begin{figure}[ht]
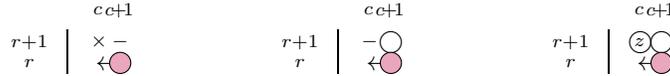

    \begin{displaymath}
      \begin{array}{rl}
        & \nulltab{ & c & c\!+\!1 } \\
        \nulltab{r+1 \\ r} & \vline\nulltab{ & \times & - \\ & & \leftball{purple}{} }
      \end{array}
      \hspace{5em}
      \begin{array}{rl}
        & \nulltab{ & c & c\!+\!1 } \\
        \nulltab{r+1 \\ r} & \vline\nulltab{ & - & \circify{ \ } \\ & & \leftball{purple}{} }
      \end{array}
      \hspace{5em}
      \begin{array}{rl}
        & \nulltab{ & c & c\!+\!1 } \\
        \nulltab{r+1 \\ r} & \vline\nulltab{ & \circify{ z } & \circify{ \ }  \\ & & \leftball{purple}{} }
      \end{array}
    \end{displaymath}
    \caption{\label{fig:commute-r}The subcases when $\Rect_c$ acts on $T$ by moving the cell in position $(c+1,r)$ left to position $(c,r)$.}
  \end{figure}
  
  \textsc{Case} ($j<r$): The cell in position $(c+1,j)$ is the lowest cell not column $c$-paired in $T$, and since $r>j$, this holds true for $\Ke_r(T)$ as well, giving $T$ and $\Ke_r(T)$ the same action of $\Rect_c$. Moreover, since $j<r$, the cells in rows $r,r+1$ coincide for $T$ and for $\Rect_c(T)$, giving them the same $r$-pairings and so, too, the same action of $\Ke_r$. Therefore $\Rect_c(\Ke_r(T)) = \Ke_r(\Rect_c(T))$ in this case as well.
\end{proof}

While rectification can be performed in basic steps, we often wish to consider the rectification of an entire column of a diagram. Extending notation, for $T$ a diagram and $c$ a column index, let $\Rect_{c}^*(D)$ denote $\Rect_{c}^m(D)$ for any (equivalently, the smallest) $m$ such that $\Rect_{c}(\Rect_{c}^m(D)) = \Rect_{c}^m(D)$. With this notation, we observe the southwest property of a diagram is preserved under rectification.

\begin{proposition}
  For $D$ a southwest diagram and $c$ any column index, $\Rect_{c}^*(D)$ is also southwest. In particular, $\rect(D)$ is a composition diagram.
  \label{prop:sw}
\end{proposition}

\begin{proof}
  A violation of the southwest property can happen under rectification in one of two ways: (a) cells $y$ above $x$ in column $c+1$ could have $y$ move left with no cell immediately left of $x$; or (b) cells $y$ above $x$ in column $c+1$ could have $x$ move left with some cell right of $x$ in its row. For (a), we must have $x$ column $c$-paired else it would move under rectification instead of $y$, and so there exists some cell $w$ in column $c$ weakly above the row of $x$. Since there is no cell immediately left of $x$, $w$ lies strictly above $x$, and so $w,x$ violate the southwest condition for $D$, a contradiction. For (b), if $z$ lies strictly right of $x$ in its row and $x$ moves left under rectification, then if any cell $y$ above $x$ in column $c+1$ is column $c$-paired, there exists a cell $w$ in column $c$ strictly above the row of $x$. In this case, $w,z$ violate the southwest condition for $D$, a contradiction. Thus rectifying all cells above $x$ restores the southwest condition, proving the first statement.

  Now consider a diagram $D$ that is both southwest and rectified. Suppose $D$ has a cell $x$ in row $r$, column $c+1$, where $c \geq 1$. Since $D$ is rectified, $x$ must be column $c$-paired with some cell $y$ in some row $s \geq r$, column $c$. Since $D$ is southwest, we must have $y$ in row $r$, column $c$. Therefore $D$ is a composition diagram. 
\end{proof}

As we shall see in Section~\ref{sec:re-label}, it often happens that $T \in \KD(D)$ but $\Rect_{c}^*(T) \not\in \KD(\Rect_{c}^*(D))$, which is to say rectification does not commute with Kohnert moves.

\subsection{Highest weights}
\label{sec:hwt}

Using the commutativity between rectification and raising operators, we can now prove the uniqueness of highest weight elements for each connected component of the Kohnert crystal of a southwest diagram by showing they rectify to composition diagrams of partition weight.

\begin{lemma}
  Let $D$ be a southwest diagram and $U \in \KD(D)$ such that $\Ke_r(U) = 0$ for all $r \geq 1$. Then $\wt(U)$ is a partition and, if $U' = \Rect_{c_k} \circ \cdots \Rect_{c_1}(U)$ is a rectified diagram, then $U'$ is a composition diagram of partition weight $\wt(U)$. In particular, $U'$ is independent of the column sequence $c_k,\ldots,c_1$.
  \label{lem:hwt-rect}
\end{lemma}

\begin{proof}
  By Lemma~\ref{lem:2.2} and the definition of the rectification operators, $\Rect_c(U') = U'$ for all column indices $c \geq 1$. By Theorem~\ref{thm:commute}, we have $\Ke_i(U') = \Ke_i(U) = 0$ for all $i \geq 0$, ensuring $U'$ is also a highest weight element. Thus, by Theorem~\ref{thm:crystal-dem}, $U'$ is a composition diagrams of partition weight. However, since rectification operators do not change the rows in which the cells lie, we must have $\wt(U)=\wt(U')$.
\end{proof}

\begin{theorem}
  Let $D$ be a southwest diagram and $T \in \KD(D)$. Then there exists a unique $U \in \KD(D)$ such that $U = \Ke_{r_m} \circ \cdots \circ \Ke_{r_1} (T)$ for some sequence of row indices $r_1,\ldots,r_m$ and such that $\Ke_r(U) = 0$ for all row indices $r \geq 1$.
  \label{thm:hwt-Kohnert}
\end{theorem}

\begin{proof}
  By Theorem~\ref{thm:closed}, for a southwest diagram $D$, the raising operators partition $\KD(D)$ into connected components. Suppose $U = \Ke_{i_m} \circ \cdots \circ \Ke_{i_1} (T)$ and $U' = \Ke_{j_n} \circ \cdots \circ \Ke_{j_1} (T)$ with $\Ke_r(U) = \Ke_r(U') = 0$ for all $r \geq 1$. By Lemma~\ref{lem:2.2}, we may apply some fixed sequence of rectification operators $R = \Rect_{c_k} \circ \cdots \Rect_{c_1}$ such that $R(T),R(U),R(U')$ are all rectified diagrams. In particular, by Lemma~\ref{lem:hwt-rect}, we have $R(U) = R(U')$. By Lemma~\ref{lem:rectification}, each of these rectification operators, and so, too, their composition $R$, commutes with $\Ke_r$ for all $r \geq 1$, and so by injectivity of raising operators, we must have $U=U'$ as desired.
\end{proof}

Next we show the rectified diagram of $T$ is is independent of the order in which the rectification operators are applied.

\begin{lemma}
  Let $T$ be a diagram. If $S = \Rect_{i_k} \circ \cdots \circ \Rect_{i_1} (T)$ and $S' = \Rect_{j_l} \circ \cdots \circ \Rect_{j_1} (T)$ are rectified diagrams for some sequences of column indices $i_1,\ldots,i_k\geq 1$ and $j_1,\ldots,j_l\geq 1$, then $S = S'$. 
  \label{lem:rectification}
\end{lemma}

\begin{proof}
  Consider first the case when $\Ke_i(T)=0$ for all $i\geq 1$. By Lemma~\ref{lem:2.2} and the definition of the rectification operators, $\Rect_c(S) = S$ and $\Rect_c(S')=S'$ for all column indices $c \geq 1$. By Theorem~\ref{thm:commute}, we have $\Ke_i(S) = \Ke_i(S') = 0$ for all $i \geq 0$, ensuring $S,S$ are highest weight elements. Thus, by Theorem~\ref{thm:crystal-dem}, $S,S'$ are composition diagrams of partition weight. However, since rectification operators do not change the rows in which the cells lie, we must have $S=S'$.

  Now let $T$ be general. By acting on $T$ iteratively with any nonzero Kohnert raising operator, we arrive at some diagram $U = \Ke_{r_m} \circ \cdots \circ \Ke_{r_1} (T)$ satisfying $\Ke_i(U)=0$ for all $i\geq 1$. Let $R = \Rect_{i_k} \circ \cdots \circ \Rect_{i_1} (U)$ and $R' = \Rect_{j_l} \circ \cdots \circ \Rect_{j_1} (U)$. By Theorem~\ref{thm:commute}, we have $R = \Ke_{r_m} \circ \cdots \circ \Ke_{r_1} (S)$ and $R' = \Ke_{r_m} \circ \cdots \circ \Ke_{r_1} (S')$. By the previous case, since $U$ is a highest weight element, we have $R=R'$, and so $\Ke_{r_m} \circ \cdots \circ \Ke_{r_1} (S) = \Ke_{r_m} \circ \cdots \circ \Ke_{r_1} (S')$. Therefore, by the injectivity of the raising operators, we have $S=S'$ as desired.
\end{proof}

By Lemma~\ref{lem:rectification}, the following is well-defined for any diagram $T$.

\begin{definition}
  Given a diagram $T$, the \newword{rectification of $T$}, denoted by $\rect(T)$, is image under $\varphi$ (Definition~\ref{def:embed-SSYT}) of the unique rectified diagram obtained by applying any terminal sequence of rectification operators to $T$.
  \label{def:embed-KT}
\end{definition}

We use rectification to embed of each component of the crystal on $\KD(D)$ into a tableaux crystal.

\begin{theorem}
  Let $D$ be a southwest diagram and $\Koh\subseteq \KD(D)$ any connected component of the Kohnert crystal on $\KD(D)$. Then rectification is a well-defined, weight-preserving injective map
  \[ \rect:\Koh \rightarrow \B(\lambda) \]
  satisfying $\rect(\Ke_i(T)) = \e_i(\rect(T))$ for all $T\in\Koh$ and all $i \geq 1$, where $\lambda$ is the unique highest weight of $\Koh$.
  \label{thm:embed}
\end{theorem}

\begin{proof}
  Rectification is well-defined by Lemma~\ref{lem:rectification}, and is weight preserving since it does not change the rows for any cell. Commutation with the Kohnert crystal operators follows from Theorem~\ref{thm:commute}, and by Theorem~\ref{thm:closed}, these operators stay within the set of Kohnert diagrams for $D$. 
\end{proof}


%
\section{Kohnert tableaux}
%
\label{sec:label}

We define labelings of diagrams to determine when a diagram $T$ lies in $\KD(D)$ for a southwest diagram $D$ without constructing the Kohnert poset on $\KD(D)$.

\subsection{Diagram labelings}
\label{sec:label-diagram}

By Theorem~\ref{thm:embed}, for $D$ southwest, each connected component of the Kohnert crystal on $\KD(D)$ embeds as a subset of a highest weight crystal $\B(\lambda)$ with highest weight $\lambda$ the unique highest weight on the component. However, the Demazure crystals $\B_w(\lambda)$ are special subsets of the full crystals $\B(\lambda)$, so it remains to show each component embeds as $\B_w(\lambda)$ for some permutation $w$.

To identify the permutation $w$, or, equivalently, the weak composition $w \cdot \lambda$, we recall the canonical labelings of rectified diagrams defined by Assaf and Searles \cite[Definition~2.3]{AS18} that can be used to identify those weak compositions $\comp{a}$ for which a rectified diagram $T$ belongs to $\KD(\D(\comp{a}))$.

\begin{definition}[\cite{AS18}]
  For a weak composition $\comp{a}$ of length $n$, a \newword{Kohnert tableau of content $\comp{a}$} is a diagram labeled with $1^{a_1}, 2^{a_2}, \ldots, n^{a_n}$ satisfying
  \begin{enumerate}[label=(\roman*)]
  \item \label{i:label} there is exactly one $i$ in each column from $1$ through $a_i$;
  \item \label{i:flag} each entry in row $i$ is at least $i$;
  \item \label{i:descend} the cells with entry $i$ weakly descend from left to right;
  \item \label{i:invert} if $i<j$ appear in a column with $i$ above $j$, then there is an $i$ in the column immediately to the right of and strictly above $j$.
  \end{enumerate}
  \label{def:kohnert-tableaux}
\end{definition}

\begin{figure}[ht]
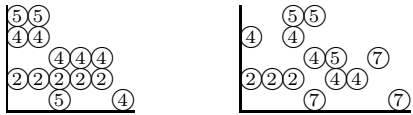

  \begin{displaymath}
    \cirtab{%
      5 & 5 \\
      4 & 4 \\
      & & 4 & 4 & 4 & \\
      2 & 2 & 2 & 2 & 2 \\
      & & 5 & & & 4 \\\hline } 
    \hspace{5\cellsize}
    \cirtab{%
      & & 5 & 5 \\
      4 & & 4 \\
      & & & 4 & 5 & & 7 \\
      2 & 2 & 2 & & 4 & 4 \\
      & & & 7 & & & & 7 \\\hline } 
  \end{displaymath}
  \caption{\label{fig:labeling}A Kohnert tableau of content $(0,5,0,6,4)$ (left) and a more general Kohnert labeling of a diagram (right).}
\end{figure}

For example, the rectified diagram on the left of Fig.~\ref{fig:labeling} is labeled in such a way that it is a Kohnert tableau. Assaf and Searles \cite[Theorem~2.8]{AS18} show for $T$ a rectified diagram, $T \in \KD(\D(\comp{a}))$ if and only if cells of $T$ can be labeled so that Definition~\ref{def:kohnert-tableaux} holds. Thus the left diagram in Fig.~\ref{fig:labeling} is a Kohnert diagram for $(0,5,0,6,4)$. We aim to create an analogous construction for $\KD(D)$ for any southwest diagram $D$, with an example shown on the right side of Fig.~\ref{fig:labeling}.

To begin, we generalize Definition~\ref{def:kohnert-tableaux}\ref{i:label} to arbitrary labelings of diagrams.

\begin{definition}
  A \newword{labeling} of a diagram $T$ is a map $\Label$ from cells of $T$ to positive integers. A labeling $\Label$ is \newword{strict} if cells in the same column have distinct labels.
  \label{def:label}
\end{definition}

By Definition~\ref{def:label}, we may regard the labels used in a strict labeling $\Label$ as a diagram $\D(\Label)$ defined by setting the cell $(c,r)$ in $\D(\Label)$ if and only if column $c$ has a cell labeled $r$. In this case, we say \newword{$T$ is labeled with respect to $\D(\Label)$}. One labeling that will be important is the \newword{super-standard labeling} of a diagram that places label $r$ in each cell of row $r$.

As we wish for the labels on the cells of $T$ to indicate from whence they came in $D$, we introduce the following condition based on Definition~\ref{def:kohnert-tableaux}\ref{i:flag}.

\begin{definition}
  A labeling $\Label$ of a diagram $T$ is \newword{flagged} if for every row $r$ and for every cell $x\in T$ lying in row $r$ we have $\Label(x) \geq r$.
  \label{def:flagged}
\end{definition}

To generalize Definition~\ref{def:kohnert-tableaux}\ref{i:descend}, we define a re-labeling procedure for diagrams as they are rectified. To do this, we give an alternative to the pairing rule used in rectification so as to take labels into account.

\begin{definition}
  Let $T$ be a diagram with labeling $\Label$. Given a column index $c$, the \newword{label $c$-pairing of $T$ with respect to $\Label$} is defined on cells of column $c+1$ as follows: assuming all cells above $x$ in column $c+1$ have been label paired, label pair $x$ with the cell weakly above it in column $c$ with the largest label that is weakly smaller than $\Label(x)$, if it exists, and otherwise leave $x$ unpaired.
  \label{def:pair-label}
\end{definition}

We use the label pairing rule to redefine labels as we rectify $T$, thereby changing the labeling diagram $D$ in the process as well.

\begin{definition}
  Let $T$ be a diagram with labeling $\Label$. Given a column $c$, the \newword{rectified labeling} $\Rect_c^*(\Label)$ of $\Rect_c^*(T)$ is constructed recursively as follows. Let $x_1,\ldots,x_m$ be the cells in column $c+1$ of $T$ that are not label $c$-paired, taken from highest to lowest. Initially set $\Label' = \Label$. Then
  \begin{itemize}
  \item for $i$ from $1$ to $m$, if there exist label $c$-paired cells $y,z$ with $z$ above $x_i$ such that $\Label(y) \leq \Label'(x_i) < \Label'(z)$, then choose $z$ so that $\Label'(z)$ is maximal and swap the labels of $x_i$ and $z$;
  \item for every cell $z$ in column $c+1$, if $z$ is label $c$-paired with $y$ in column $c$, then set $\Label'(z) = \Label(y)$.
  \end{itemize}
  Then $\Rect_c^*(\Label)$ is the labeling of $\Rect_c^*(T)$ obtained by maintaining labels from $\Label'$ as cells move left under rectification.
  \label{def:re-label}
\end{definition}

\begin{figure}[ht]
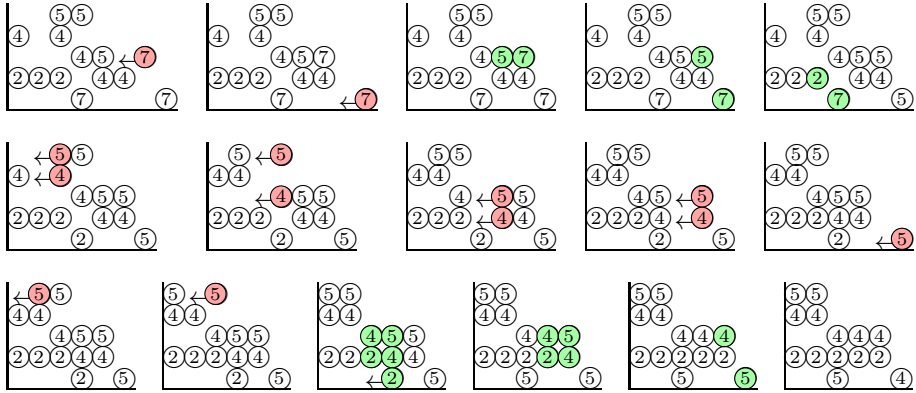

  \begin{displaymath}
    \begin{array}{l}
      \cirtab{%
        & & 5 & 5 \\
        4 & & 4 \\
        & & & 4 & 5 & & \leftball{red}{7} \\
        2 & 2 & 2 & & 4 & 4 \\
        & & & 7 & & & & 7 \\\hline } \hspace{1.4\cellsize}
      \cirtab{%
        & & 5 & 5 \\
        4 & & 4 \\
        & & & 4 & 5 & 7 \\
        2 & 2 & 2 & & 4 & 4 \\
        & & & 7 & & & & \leftball{red}{7} \\\hline } \hspace{1.4\cellsize}
      \cirtab{%
        & & 5 & 5 \\
        4 & & 4 \\
        & & & 4 & \cball{green}{5} & \cball{green}{7} \\
        2 & 2 & 2 & & 4 & 4 \\
        & & & 7 & & & 7 \\\hline } \hspace{1.4\cellsize}
      \cirtab{%
        & & 5 & 5 \\
        4 & & 4 \\
        & & & 4 & 5 & \cball{green}{5} \\
        2 & 2 & 2 & & 4 & 4 \\
        & & & 7 & & & \cball{green}{7} \\\hline } \hspace{1.4\cellsize}
      \cirtab{%
        & & 5 & 5 \\
        4 & & 4 \\
        & & & 4 & 5 & 5 \\
        2 & 2 & \cball{green}{2} & & 4 & 4 \\
        & & & \cball{green}{7} & & & 5 \\\hline } \\ \\
      \cirtab{%
        & & \leftball{red}{5} & 5 \\
        4 & & \leftball{red}{4} \\
        & & & 4 & 5 & 5 \\
        2 & 2 & 2 & & 4 & 4 \\
        & & & 2 & & & 5 \\\hline } \hspace{2.4\cellsize}
      \cirtab{%
        & 5 & & \leftball{red}{5} \\
        4 & 4 & \\
        & & & \leftball{red}{4} & 5 & 5 \\
        2 & 2 & 2 & & 4 & 4 \\
        & & & 2 & & & 5 \\\hline } \hspace{2.4\cellsize}
      \cirtab{%
        & 5 & 5 \\
        4 & 4 & \\
        & & 4 & & \leftball{red}{5} & 5 \\
        2 & 2 & 2 & & \leftball{red}{4} & 4 \\
        & & & 2 & & & 5 \\\hline } \hspace{1.4\cellsize}
      \cirtab{%
        & 5 & 5 \\
        4 & 4 & \\
        & & 4 & 5 & & \leftball{red}{5} \\
        2 & 2 & 2 & 4 & & \leftball{red}{4} \\
        & & & 2 & & & 5 \\\hline } \hspace{1.4\cellsize}
      \cirtab{%
        & 5 & 5 \\
        4 & 4 & \\
        & & 4 & 5 & 5  \\
        2 & 2 & 2 & 4 & 4  \\
        & & & 2 & & & \leftball{red}{5} \\\hline } \\ \\
      \cirtab{%
        & \leftball{red}{5} & 5 \\
        4 & 4 & \\
        & & 4 & 5 & 5  \\
        2 & 2 & 2 & 4 & 4  \\
        & & & 2 & & 5 \\\hline } \hspace{1.3\cellsize}
      \cirtab{%
        5 & & \leftball{red}{5} \\
        4 & 4 & \\
        & & 4 & 5 & 5  \\
        2 & 2 & 2 & 4 & 4  \\
        & & & 2 & & 5 \\\hline } \hspace{1.3\cellsize}
      \cirtab{%
        5 & 5 \\
        4 & 4 & \\
        & & \cball{green}{4} & \cball{green}{5} & 5  \\
        2 & 2 & \cball{green}{2} & \cball{green}{4} & 4  \\
        & & & \leftball{green}{2} & & 5 \\\hline } \hspace{1.3\cellsize}
      \cirtab{%
        5 & 5 \\
        4 & 4 & \\
        & & 4 & \cball{green}{4} & \cball{green}{5}  \\
        2 & 2 & 2 & \cball{green}{2} & \cball{green}{4}  \\
        & & 5 & & & 5 \\\hline } \hspace{1.3\cellsize}
      \cirtab{%
        5 & 5 \\
        4 & 4 & \\
        & & 4 & 4 & \cball{green}{4}  \\
        2 & 2 & 2 & 2 & 2  \\
        & & 5 & & & \cball{green}{5} \\\hline } \hspace{1.3\cellsize}
      \cirtab{%
        5 & 5 \\
        4 & 4 & \\
        & & 4 & 4 & 4  \\
        2 & 2 & 2 & 2 & 2  \\
        & & 5 & & & 4 \\\hline } 
    \end{array}
  \end{displaymath}
  \caption{\label{fig:re-labeling}The rectified relabeling of a diagram, with cells that get re-labeled highlighted (along with their label pairs) and cells that move under rectification indicated with an arrow.}
\end{figure}

For example, Fig.~\ref{fig:re-labeling} shows the steps by which the right diagram in Fig.~\ref{fig:labeling} relabels to the left, giving a refinement of rectification shown in Fig.~\ref{fig:rectification}.

If $T$ has the same column weight as $D$, then we can label $T$ with respect to $D$ and rectify the labeling. Since the order of rectification does not affect the result, we henceforth rectify from right to left. With this in mind, we have the following generalization of Definition~\ref{def:kohnert-tableaux}\ref{i:descend}.

\begin{definition}
  A labeling $\Label$ of a diagram $T$ is \newword{semi-proper} if the cells with entry $i$ in the fully rectified labeling $\rect(\Label)$ of the rectified diagram $\rect(T)$ weakly descend from left to right.
  \label{def:semiproper}
\end{definition}

To generalize Definition~\ref{def:kohnert-tableaux}\ref{i:invert}, we generalize the \emph{labeling algorithm} defined by Assaf and Searles for rectified diagrams \cite[Definition~2.5]{AS18} that greedily assigns labels so that Definition~\ref{def:kohnert-tableaux}\ref{i:descend} holds.

\begin{definition}[\cite{AS18}]
  Given a weak composition $\comp{a}$ and a diagram $T\in\KD(\D(\comp{a}))$, the \newword{Kohnert labeling of $T$ with respect to $\comp{a}$}, denoted by $\Label_{\comp{a}}(T)$, assigns labels to cells of $T$ as follows. Assuming all columns right of column $c$ have been labeled, assign labels $\{i \mid a_i \geq c\}$ to cells of column $c$ from bottom to top by choosing the smallest label $i$ such that the $i$ in column $c+1$, if it exists, is weakly lower.
\label{def:labelling_algorithm}
\end{definition}

Assaf and Searles \cite[Theorem~2.8]{AS18} prove for any diagram $T$ and weak composition $\comp{a}$ for which $T$ and $\D(\comp{a})$ have the same column weight, the Kohnert labeling of $T$ with respect to $\comp{a}$ is well-defined and flagged for $T$ if and only if $T\in\KD(\D(\comp{a}))$. This provides the bijection between Kohnert tableaux and Kohnert diagrams for $\comp{a}$. In this case, for any cell $x$ in column $c+1$, there exists a cell $y$ in column $c$ with $\Label_{\comp{a}}(y) = \Label_{\comp{a}}(x)$ and $y$ lies weakly above $x$. Thus $x$ will label $c$-pair with $y$, ensuring $\Label_{\comp{a}}(T)$ is invariant under re-labeling. 

If $T$ has the same column weight as $D$, then we can label $T$ with respect to $D$ by a greedy algorithm generalizing Definition~\ref{def:labelling_algorithm} by forcing the labeling to be semi-proper. For this, given a diagram $T$ and a column $c$, we partition $T$ into $T_{\leq c} \sqcup T_{>c}$, where the former contains all cells in columns weakly left of $c$, and the latter contains all cells in columns strictly right of $c$. When we rectify $T_{>c}$, we still regard all cells as being weakly right of column $c$.

\begin{definition}
  For diagrams $T,D$ of the same column weight, construct the \newword{Kohnert labeling of $T$ with respect to $D$}, denoted by $\Label_{D}(T)$, as follows. Once all columns of $T$ right of column $c$ have been labeled, set $\displaystyle T' = T_{\leq c} \ \sqcup \ \rect(T_{>c})$, where the leftmost occupied column of the latter is $c+1$. Bijectively assign labels
  \[\{r \mid D \text{ has a cell in column $c$, row $r$ } \}\]
  to cells in column $c$ of $T$ from smallest to largest by assigning label $r$ to the lowest unlabeled cell $x$ such that if there exists a cell $z$ in column $c+1$ of $\rect(T_{>c})$ with label $r$, then $x$ lies weakly above $z$.
  \label{def:labeling}
\end{definition}

We generalize Definition~\ref{def:kohnert-tableaux} to a notion of flagged, proper labelings.

\begin{definition}
  A labeling $\Label$ of a diagram $T$ is \newword{proper} if $\Label = \Label_D$ for some $D$.
  \label{def:proper}
\end{definition}

Since Fig.~\ref{fig:re-labeling} rectified from right to left, we see this is a proper labeling.

Notice a proper labeling is, in particular, strict. Henceforth we shall rely on the southwest property of diagrams, and so we say a proper labeling $\Label_D$ is southwest if and only if the diagram $D$ is southwest.

\subsection{Kohnert labelings}
\label{sec:label-kohnert}

The existence of a strict, semi-proper, southwest labeling for a diagram implies the existence of a proper labeling.

\begin{lemma}
  Let $\Label$ be a semi-proper labeling of $T$ with respect to a southwest diagram $D$. Then $\Label_D$ is well-defined on $T$, and if $\Label$ is flagged, then so is $\Label_D$.
  \label{lem:semi-proper}
\end{lemma}

\begin{proof}
  Suppose $\Label_D$ is well-defined on and agrees with $\Label$ on all columns strictly to the right of column $c$. If $c=0$, then the result is trivial. Thus we may proceed by induction on $c$. Suppose $\Label_D$ is well-defined on and agrees with $\Label$ on all columns strictly to the right of column $c$, and consider column $c$ itself. Define a new labeling $\Label'$ on $T$ as follows. Initially set $\Label' = \Label$, which we note also coincides with $\Label_D$ for columns strictly to the right of column $c$. Beginning with the smallest label, for each label $i$, let $s$ denote the row of $i$ in column $c$ of $\Label'$ and let $r \leq s$ denote the lowest row with label in $\Label'$ larger than $i$ for which placing $i$ into row $r$ remains semi-proper. If $r<s$, then swap labels $i$ and $j$ in column $c$ of $\Label'$. This will still be semi-proper with columns to the right since $i$ has chosen its cell to maintain that property and since $j$ has moved up. Since $D$ is southwest, as we rectify, there must be at least as many columns left of $c$ with label $i$ as have label $j$ since $i<j$ occurs in column $c$. Thus we may permute the rows of these labels in those columns as well, thereby maintaining the semi-proper property to the left of column $c$. Once we complete this process with each label in column $c$, $\Label'$ will be semi-proper and coincide with $\Label_D$ for columns strictly to the right of column $c-1$. Thus, by induction, $\Label_D$ is well-defined for $T$. When a label $j$ moves up, it is displaced by a smaller label $i$. Thus if $\Label$ is flagged, then $i$ is at least as great as its row in $\Label$, and so $j>i$ will be at least as great as its row in $\Label'$. Hence $\Label'$ is also flagged, and so $\Label_D$ is as well.
\end{proof}

Our first result justifying the slew of definitions for labelings is the following.

\begin{theorem}
  For $D$ a southwest diagram, if $T\in\KD(D)$, then $\Label_{D}$ is well-defined and flagged for $T$.
  \label{thm:label-necessary}
\end{theorem}

\begin{proof}
  For $D$ itself, the labeling algorithm places label $r$ in each cell of row $r$, so $\Label_{D}(D)$ is well-defined and flagged. Thus we may proceed by induction on the minimum number of Kohnert moves needed to obtain a diagram from $D$. Suppose $T$ is obtainable from some diagram $S\in\KD(D)$ by a single Kohnert move, say by moving down the cell $x$ in column $c$, row $r$, where the theorem holds for $S$. Let $\Label$ be the labeling on $T$ defined for each column by taking the labels in the same column order as for $\Label_D$ on $S$. Notice $\Label$ is semi-proper on $T$ for all columns strictly to the right of column $c$, and $\Label$ is flagged on $T$.

  If $x$ moves down one row, then we claim $\Label$ is semi-proper. This holds for column $c$ since $x$ has moved down and there is no cell to the right in the row from whence it came (since moving $x$ down is a Kohnert move), and it holds for columns strictly to the left of column $c$ since $x$ has moved down in column $c$ and no other labels have moved. Therefore, by Lemma~\ref{lem:semi-proper}, $\Label_D$ is well-defined and flagged on $T$.

  Thus we may assume both $S$ and $T$ have a cell $y$ in column $c$, row $r-1$ over which $x$ jumps. Suppose first there is no cell in column $c$ row $r-2$. If $\Label_D(y) > \Label_D(x)$ in $S$, then since $x$ is above $y$, by Definition~\ref{def:kohnert-tableaux}\ref{i:invert} in $\rect(S_{>c})$ there must be a cell $z$ with $\Label_D(z)=\Label_D(x)$ in column $c+1$ weakly below $x$ and strictly above $y$, and hence in row $r$, contradicting that going from $S$ to $T$ is a Kohnert move. Therefore $\Label_D(y) < \Label_D(x)$. Let $z$ be the cell of $\rect(S_{>c})$ in column $c+1$ with $\Label_D(z) = \Label_D(y)$, if it exists. If $z$ lies in row $r-1$, then in $T$, we regard $x$ as jumping over $y$ so that, in $\Label$, the labels remain the same, with $x$ now below $y$ in row $r-2$; otherwise, in $T$, we regard both $x$ and $y$ as moving down one row so that, in $\Label$, the labels remain the same with $x$ in row $r-1$ and $y$ in row $r-2$. Then $\Label$ semi-proper and flagged in column $c$, and so too to the left since labels have moved only downward. Again, by Lemma~\ref{lem:semi-proper}, $\Label_D$ is well-defined and flagged on $T$.

  Now let $r'<r-2$ denote the highest empty position in column $c$ below row $r$, which is the position into which the cell $x$ lands in passing from $S$ to $T$. Let $y_1,\ldots,y_{r-r'-1}$ denote the cells in column $c$ in rows $r'+1,\ldots,r-1$ taken highest to lowest. We claim $\Label_D(y_h)<\Label_D(x)$ in $S$ for each $h$. This is true for $y_1$ by the previous case, so we proceed by induction, assuming $\Label_D(y_h)<\Label_D(x)$ for each $h<n$. Suppose, for contraction, $\Label_D(y_n)>\Label_D(x)$. By Definition~\ref{def:kohnert-tableaux}\ref{i:invert} applied to each cell above up through row $r$, there must be cells $z,z_1,\ldots,z_{n-1}$ in column $c+1$ of $\rect(S_{>c})$ with $\Label_D(z)=\Label_D(x)$ and $\Label_D(z_h)=\Label_D(y_h)$ and each lies strictly above $y_n$ and weakly below the cell in column $c$ with the same label. Since these cells are ordered $z,z_1,\ldots,z_{n-1}$ from top to bottom, this places $z$ in row $r$, contradicting that going from $S$ to $T$ is a Kohnert move. Thus the claim is proved. Now we may repeat the argument of the previous case, where in $T$ we regard each $y_i$ either as remaining in place if there is another cell in the same row with the same label or as moving down one row otherwise. The same analysis applies, showing once again $\Label_D$ is well-defined and flagged on $T$.
\end{proof}

Label $c$-pairing and column $c$-pairing are compatible for proper labelings coming from southwest diagrams, and so in this case rectified cells are never label $c$-paired.

\begin{lemma}
  For $\Label$ a proper, southwest labeling of $T$, a cell $x$ in column $c+1$ of $T$ is column $c$-paired if and only if it is label $c$-paired with respect to $\Label$.
  \label{lem:same-move}
\end{lemma}

\begin{proof}
  Suppose the result holds for all cells (possibly none) above $x$ in column $c+1$. If $x$ is label $c$-paired, say to a cell $y$ in column $c$, then $y$ lies weakly above $x$ and is not label $c$-paired in $T$ with a cell above $x$, and hence $y$ is not column $c$-paired in $T$ with a cell above $x$. Thus $x$ is column $c$-paired in $T$.

  Conversely, suppose $x$ is not label $c$-paired in $T$. Then every cell $y$ in column $c$ weakly above $x$ with $\Label(y)\leq \Label(x)$ is label $c$-paired with some cell above $x$. If there is no cell $y$ in column $c$ weakly above $x$ with $\Label(y)> \Label(x)$, then all cells above $x$ in column $c$ are column $c$-paired with a cell above $x$, and so $x$ is not column $c$-paired. Thus assume there exists a cell $y$ above $x$ in column $c$ with $\Label(y)> \Label(x)$. To arrive at a contradiction, we construct an infinite sequence of distinct cells $x_0,x_1,x_2\ldots$ in column $c+1$ and $y_0,y_1,\ldots$ in column $c$ such that:
  \begin{itemize}
  \item[(a)] $\Label(y_i) = \Label(x_i) < \Label(y)$,
  \item[(b)] $y_i,x_i$ lie strictly below $y$,
  \item[(c)] $y_i$ is label $c$-paired with $x_{i+1}$.
  \end{itemize}
  To begin, we simply take $x=x_0$. Now assume, for some $m\geq 0$, we have constructed $x_0,\ldots,x_m$ and $y_0,\ldots,y_{m-1}$ for which these properties hold, and we will construct $y_m$ and $x_{m+1}$. Since $\Label$ is southwest, and since, by (a), $\Label(x_m) < \Label(y)$, as we rectify there must exist a cell $y_m$ in column $c$ with $\Label(y_m) = \Label(x_m)$, proving (a) for $y_m$. Since $y$ lies above $x_m$ and $\Label(y) > \Label(y_m)$, the labeling algorithm ensure $y_m$ lies below $y$ since $y_m$ is labeled before $y$, proving (b) for $y_m$. The labeling algorithm also places $y_m$ weakly above $x_m$ since they share the same label, and so for $x_m$ not to have chosen $y_m$ for its label $c$-pair, there must be some cell, say $x_{m+1}$, above $x_m$ in column $c+1$ to which $y_m$ is label $c$-paired, proving (c) for both $y_m,x_{m+1}$. By the label pairing rule, $x_{m+1}$ is weakly below $y_m$ and so below $y$, proving (b) for $x_{m+1}$. Finally, since $\Label(y_m) < \Label(y)$, the label pairing rule ensures $\Label(y) > \Label(x_{m+1})$, since otherwise $x_{m+1}$ would select $y$ over $y_m$ for its label $c$-pair, thereby proving (a) for $x_{m+1}$. We may continue indefinitely, creating an infinite sequence of distinct cells thereby contradicting the finiteness of the diagram $T$. Therefore no such $y$ exists, and so $x$ is not column $c$-paired in $T$. 
\end{proof}

For $D$ southwest, Theorem~\ref{thm:label-necessary} gives a necessary condition for an arbitrary diagram $T$ to be in $\KD(D)$ by testing if the labeling algorithm for $D$ is well-defined and flagged on $T$. Amazingly, this is also sufficient. 

\begin{theorem}
  For $D$ a southwest diagram, if $T$ is a diagram for which $\Label_{D}$ is well-defined and flagged, then $T\in\KD(D)$.
  \label{thm:label-sufficient}
\end{theorem}

\begin{proof}
  Whenever $\Label_D$ is well-defined on $T$, we can consider the \emph{label displacement} defined by $\delta_D(T) = \sum_{x \in T} \Label_D(x) - \mathrm{row}(x)$, and this will be nonnegative whenever $\Label_D$ is flagged on $T$. We prove the result by induction on this statistic. If $\delta_D(T) = 0$, then we must have $\Label_D(x) = \mathrm{row}(x)$ for every cell $x$ of $T$, and so $T=D \in \KD(D)$. Suppose $\Label_D$ is well-defined and flagged on some diagram $T$ with $\delta_D(T)>0$. We will construct a diagram $S$ such that $T$ is obtained from a Kohnert move on $S$ and $\Label_D$ is well-defined and flagged on $S$. Since the labels within each column are the same for $S,T$, but cells are moved down in $S$, we have $\delta_D(S) < \delta_D(T)$. Thus by induction $S\in\KD(D)$, and so $T\in\KD(S)\subseteq\KD(D)$. 
  
  Suppose $\rect(T_{\geq c})$ is a composition diagram for every column $c$. Since $\Label_D$ is flagged on $T$ and $\delta_D(T)>0$, there exists a cell $x$ with $\Label_D(x) > \mathrm{row}(x)$. Consider cells $x$ with $\Label_D(x) > \mathrm{row}(x)$ such that for all cells $y$ left of $x$ with $\Label_D(y) = \Label_D(x)$, we have $\mathrm{row}(y) = \Label_D(x)$. Among all such cells, choose $x$ with maximal label, say in row $r$, column $c$. Let $s\geq r$ be such that all rows $r,r+1,\ldots,s$ have a cell in column $c$ and there is no cell in row $s+1$, column $c$. Since $\rect(T_{\geq c}))$ is a composition diagram, there is no cell weakly to the right of column $c$ in row $s+1$, and so we may lift $x$ to row $s+1$ by a reverse Kohnert move to construct the diagram $S$. Let $\Label$ be the labeling on $S$ induced from $\Label_D(T)$ by maintaining all labels, including for $x$. By choice of $x$, the cells in rows $r+1,\ldots,s$, column $c$ must have label strictly less than $\Label_D(x)$, and so $\Label_D(x) \geq s+1$. Thus $\Label$ is flagged for $S$. We claim $\Label$ is semi-proper on $S$. By Definition~\ref{def:kohnert-tableaux}\ref{i:descend}, for each row $r+1,\ldots,s$, there must be another cell to the right of column $c$ that label pairs to the cell in the same row in column $c$, and again by choice of $x$ these cells have smaller label than $\Label_D(x)$. Thus the labeling algorithm in column $c$ of $S$ will place $\Label_D(x)$ into $x$, ensuring $\Label = \Label_D$ weakly to the right of column $c$, and so it is semi-proper there. Looking left, if there is no cell $y$ with $\Label(y)=\Label(x)$ in the column to the left, then the labeling remains semi-proper. If at some point $x$ is label pairs with a cell in row $s+1$, then thereafter the labeling remains semi-proper since it agrees with $\Label_D(Y)$ to the left. If neither of these is the case, then we must have encountered a cell $y$ with $\Label(y)=\Label(x)$, which by choice of $x$ lies in row $\Label(x) \geq s+1$, and so must be weakly above row. Thus $\Label$ is semi-proper and flagged on $S$, so by Lemma~\ref{lem:semi-proper}, $\Label_D$ is well-defined and flagged on $S$, proving the theorem in this case.

  Otherwise, choose $c$ so that $\rect(T_{\geq c'})$ is a composition diagram for every column $c'>c$, but $\rect(T_{\geq c})$ is not a composition diagram. In this case there exists a cell in column $c$ that label $(c-1)$-pairs with a cell in column $c-1$ that lies strictly higher. Let $x$ be the highest such cell, say in column $c$, row $r$. AS before, let $s\geq r$ be such that all rows $r,r+1,\ldots,s$ have a cell in column $c$ and there is no cell in row $s+1$, column $c$. Since $\rect(T_{> c}))$ is a composition diagram, there is no cell weakly to the right of column $c$ in row $s+1$, and so we may lift $x$ to row $s+1$ by a reverse Kohnert move to construct the diagram $S$. Let $\Label$ be the labeling on $S$ induced from $\Label_D(T)$ by maintaining all labels, including for $x$. Since $S_{>c} = T_{>c}$, $\Label(S) = \Label_D(S)$ strictly to the right of column $c$. In column $c$ of $T$, let $x_{i}$ denote the cell in row $r+i$ for $i=0,\ldots,s-r$. Define $0=i_0<i_1<\cdots<i_m$ by, for $j>0$, taking $x_{i_j}$ to be the lowest cell above $x_{i_{j-1}}$ and weakly below row $s$ with $\Label_D(x_{i_j})>\Label_D(x_{i_{j-1}})$. In $S$, let $i_{m+1}=s+1$ and shift labels by setting $\Label(x_{i_{j+1}}) = \Label_D(x_{i_{j}})$ for $0\leq j \leq m$, so that $\Label(S) = \Label_D(S)$ in column $c$ as well. Moreover, since $\Label_D$ is flagged on $T$, we have $\Label(x_{i_{j+1}}) = \Label_D(x_{i_{j}}) > \Label_D(x_i) \geq \mathrm{row}(r+i)$ for all $i_j < i < i_{j+1}$, showing so $\Label$ is flagged on $S$ in column $c$, and so $\Label$ is flagged on all of $S$. 

  Looking left, by choice of $x$, each cell in column $c$, row $r+1,\ldots,s$ must either label $(c-1)$-pair with the cell immediately to its left or not label $(c-1)$-pair at all. If $x_i$ is the lowest cell in column $c$ above $x$ and weakly below row $s$ not label $(c-1)$-paired, then by Lemma~\ref{lem:same-move}, neither is it column $(c-1)$-paired, in which case $x$ is not column $(c-1)$-paired either, a contradiction. Thus for $1 \leq i \leq s-r$, we may let $y_i$ in row $r+i$, column $c$ be the cell to which each $x_i$ is label $(c-1)$-paired in $T$, and let $y$ denote the cell in column $c-1$ to which $x$ is label $(c-1)$-paired which necessarily lies weakly above row $s+1$. By the southwest condition, there is a cell with label $\Label_D(x_{i_j})$ in column $c-1$ of $T$ for every $j=0,\ldots,{m-1}$, and by Definition~\ref{def:kohnert-tableaux}\ref{i:descend}, it lies weakly above row $r+i_j$. Since $x$ does not label $(c-1)$-pair with the cell immediately to its left, if it exists, this ensures $\Label_D(y_{i_{j-1}}) = \Label_D(x_{i_{j}})$. Therefore in $\Label(S)$, every cell in column $c$ not in row $s+1$ label $(c-1)$-pairs with the same cell in column $c-1$ to which it label $(c-1)$-pairs in $T$, and the cell in column $c$, row $s+1$ label $(c-1)$-pairs with $y$. Thus $\Label(S)=\Label_D(S)$ for all columns, showing $\Label_D$ is well-defined and flagged on $S$. The theorem now follows.
\end{proof}

Combining Theorems~\ref{thm:label-necessary} and \ref{thm:label-sufficient}, for $D$ southwest, $\Label_D$ is well-defined and flagged for $T$ if and only if $T \in \KD(D)$.

\subsection{Rectified labelings}
\label{sec:re-label}

We turn our attention now to showing rectification of labelings commutes with the Kohnert crystal operators while preserving the proper and flagged properties of the labelings. 

\begin{lemma}
  Let $T$ be a diagram with labeling $\Label$. For every column $c$, $\Label$ is flagged if and only if $\Rect_c^*(\Label)$ is flagged.
  \label{lem:flagged}
\end{lemma}

\begin{proof}
  Suppose $\Label$ is flagged, and consider cells in column $c+1$. If a cell $z$ in column $c+1$ label $c$-pairs with a cell $y$ in column $c$, then $y$ is weakly above $z$ and $\Label(y) \leq \Label(z)$. Thus changing the label of $z$ to match that of $y$ maintains the flagged condition for $z$ since $y$ is flagged and $z$ lies weakly below with the same label. Any cell $x$ in column $c+1$ that is not label $c$-paired will be re-labeled with the label of some $z$ above $x$ with $\Label(z) > \Label(x)$, so once again the flagged condition is maintained if it holds for $\Label$ since the larger label moves down.

  Suppose $\Label$ is not flagged with some cell $x$ in column $c+1$ satisfying $\Label(x) < \mathrm{row}(x)$. If $x$ is label $c$-paired with $y$ in column $c$, then $\mathrm{row}(y) \geq \mathrm{row}(x) > \Label(x) \geq \Label(y)$ is a violation of the flagged condition in column $c$ that is maintained by re-labeling. If $x$ is not label $c$-paired, then if it does not get relabeled, of course it creates a violation in $\Rect_c^*(\Label)$ as well. If it does get relabeled, say by some $z$ in column $c+1$, then $\Label(z) > \Label(x) > \mathrm{row}(x)$, so changing the label maintains the violation of the flagged condition at $x$.
\end{proof}

To improve upon the injective map induced by rectification, we show the re-labeled rectification of a proper, flagged labeling is a Kohnert tableau.

\begin{theorem}
  For $\Label$ a proper, southwest labeling of a diagram, $\D(\rect(\Label))$ is a composition diagram, and $\rect(\Label)$ is a Kohnert tableau if and only if $\Label$ is flagged.
  \label{thm:stable}
\end{theorem}

\begin{proof}
  By Lemma~\ref{lem:flagged}, $\rect(\Label)$ is flagged if and only if $\Label$ is flagged, so we need only show the labeling $\rect(\Label)$ on the rectified diagram $\rect(T)$ satisfies Definition~\ref{def:kohnert-tableaux}\ref{i:label}, \ref{i:descend}, and \ref{i:invert}. We proceed by induction on the number of columns. If $T,D$ have one column, then $\Label$ labels the cells of $T$ increasing from bottom to top, and so \ref{i:label} labels are distinct, \ref{i:descend} labels do not appear in consecutive columns, and \ref{i:invert} there are no column inversions. Thus we may assume the result for all diagrams with fewer columns than $T$. In particular, the result holds for $T_{>1}$, and so $\rect(T_{>1})$ satisfies Definition~\ref{def:kohnert-tableaux}\ref{i:label}, \ref{i:descend}, and \ref{i:invert}. Consider $T_1 \sqcup \rect(T_{>1})$. 

  For Definition~\ref{def:kohnert-tableaux}\ref{i:label}, we must show each step of re-labeling and rectification maintains the strictness of the labeling. Consider cells $y,x$ in columns $c,c+1$ with $\Label(y)=\Label(x)$. By the labeling algorithm (if $c=1$) or Definition~\ref{def:kohnert-tableaux}\ref{i:descend} (if $c>1$), $y$ must lie weakly above $x$. If $x$ is label $c$-paired with $y$, then by Lemma~\ref{lem:same-move}, $x$ is column $c$-paired and so does not move into column $c$. If $x$ does not label $c$-pair with $y$, then there exists some $z$ strictly above $x$ in column $c+1$ and weakly below $y$ with $\Label(z) > \Label(x)$. By the re-labeling procedure, $x$ will be re-labeled before it moves left, and so the resulting labeling is strict. 

  Definition~\ref{def:kohnert-tableaux}\ref{i:descend} holds for columns $1$ and $2$ by the labeling algorithm and for columns $2$ and beyond by induction. When cells are re-labeled for rectification, any cell in column $c$ that is label $c$-paired with a cell in column $c+1$ lies weakly above it by construction. Any cell that is not label $c$-paired moves left by Lemma~\ref{lem:same-move}, and so there will not be a cell in column $c+1$ with the same label afterward. Thus Definition~\ref{def:kohnert-tableaux}\ref{i:descend} holds throughout rectification.

  For Definition~\ref{def:kohnert-tableaux}\ref{i:invert}, suppose $y$ lies above $x$ in column $1$ with $\Label(y) < \Label(x)$. By the labeling algorithm, this happens only if $\Label(y)$ preferred not to occupy the lower cell $x$, which ensures there exists a cell $z$ in column $2$ weakly below $y$ and strictly above $x$ with $\Label(z) = \Label(y)$. When cells are re-labeled for rectification, any label $c$-paired cell in column $c+1$ gets a weakly smaller label, and so cannot create a new violation of \ref{i:invert}. Any cell $x$ in column $c+1$ not label $c$-paired will move left by Lemma~\ref{lem:same-move}, taking with it the largest label from $z$ above so that the the cell $y$ to which $z$ is label $c$-paired satisfies $\Label(y) < \Label(z)$. Since $z$ will be relabeled with $\Label(y)$, Definition~\ref{def:kohnert-tableaux}\ref{i:invert} is preserved.
\end{proof}

Rectified labels commute with crystal operators, ultimately allowing us to identify the permutation $w$ for the Demazure crystal.

\begin{lemma}
  Let $D$ be a southwest diagram, and $T$ a diagram for which $\Ke_r(T) \neq 0$. Then $\Label_D$ is well-defined for $T$ if and only if $\Label_D$ is well-defined for $\Ke_r(T)$, and in this case $\D(\rect(\Label_D(T))) = \D(\rect(\Label_D(\Ke_r(T))))$. Moreover, if $\Label_D(T)$ is flagged, then $\Label_D(\Ke_r(T))$ is flagged. 
  \label{lem:label-crystal}
\end{lemma}

\begin{proof}
  Suppose $x$ in row $r+1$, column $c$ is the cell that moves down when $\Ke_r$ acts on $T$, and suppose $\Label_D$ is well-defined on at least one of $T$ and $\Ke_r(T)$. Since $T$ and $\Ke_r(T)$ coincide for all cells strictly to the right of column $c$, the labeling algorithm proceeds the same for both up to column $c$, so both are well-defined there.

  We claim $T$ and $\Ke_r(T)$ are labeled the same by $\Label_D$ in column $c$ as well, where the label of $x$ is maintained as it moves. If there is no cell in column $c+1$, row $r+1$, then claim follows since $\rect(T_{>c}) = \rect(\Ke_r(T)_{>c})$ and so the set of labels weakly below $x$ is the same in both. Otherwise, let $n \geq 1$ be maximal such that there exist $z_1,\ldots,z_n$ in row $r+1$ with $z_i$ in column $c+i$. Since $\Ke_r$ acts on $x$, each cell $z_i$ must be $r$-paired with some cell $y_i$ in row $r$ weakly left of $z_i$ and, since $x$ is not $r$-paired, strictly right of $x$. Thus there exist $y_1,\ldots,y_n$ in row $r$ with $y_i$ in column $c+i$; see Fig.~\ref{fig:to-right}. If $\Label_D(y_i) > \Label_D(z_i)$, then, by Theorem~\ref{thm:stable} Definition~\ref{def:kohnert-tableaux}\ref{i:invert} applies, and so there exists $z$ with $\Label_D(z) = \Label_D(z_i)$ in column $c+i+1$ weakly below $z_i$ and strictly above $y_i$, forcing $z$ into row $r+1$. If $i<n$, then $z=z_{i+1}$, but if $i=n$, then this contradicts the choice of $m$, so we must have $\Label_D(y_n) < \Label_D(z_n)$. Thus there exists $m \leq n$ such that $\Label_D(z_1) = \cdots = \Label_D(z_{m})$ and $\Label_D(y_{m}) < \Label_D(z_{m}) < \Label_D(y_i)$ for $i < m$. If $m>1$, then since $D$ is southwest, there exists a cell $w$ in column $c+m-1>c$ with $\Label_D(w) = \Label_D(y_m)$. By the labeling algorithm $w$ lies weakly above $y_m$ and every cell below $w$ and weakly above $y_m$ has smaller label, contracting that $\Label_D(y_{m-1}) > \Label_D(y_{m}) = \Label_D(w)$. Thus we must have $m=1$. 

  \begin{figure}[ht]
    \begin{center}
      \begin{tikzpicture}[xscale=0.45,yscale=0.45]
        \node at (1,2) (c1) {$\scriptstyle c$};
        \node at (-0.5,1) (r2) {$\scriptstyle r+1$};
        \node at (-0.5,0) (r1) {$\scriptstyle r$};
        \node at (1,1) (x)  {$\bigcir{x}$};
        \node at (1,0) (x1)  {$-$};
        \node at (2,1) (z1)  {$\bigcir{z_1}$};
        \node at (4,1) (zm)  {$\bigcir{z_m}$};
        \node at (6,1) (zn)  {$\bigcir{z_n}$};
        \node at (7,1) (z)  {$-$};
        \node at (2,0) (y1)  {$\bigcir{y_1}$};
        \node at (4,0) (ym)  {$\bigcir{y_m}$};
        \node at (6,0) (yn)  {$\bigcir{y_n}$};
        \draw[thick,dotted] (z1) -- (zm);  
        \draw[thick,dotted] (zm) -- (zn);  
        \draw[thick,dotted] (y1) -- (ym);  
        \draw[thick,dotted] (ym) -- (yn);  
      \end{tikzpicture}
      \caption{\label{fig:to-right}Situation in $T$ when $\Ke_r$ acts on $x$, where $\Label_D(z_1) = \cdots = \Label_D(z_{m})$ and $\Label_D(y_{m}) < \Label_D(z_{m}) < \Label_D(y_{i})$ for $i<m$.}
    \end{center}
  \end{figure}

  Now consider $\Label_D(x)$ for whichever of $T$ or $\Ke_r(T)$ the labeling is well-defined. If $\Label_D(z_1) > \Label_D(x)$, then $\Label_D$ is well-defined on and agrees on both $T$ and $\Ke_R(T)$ since the labeling algorithm does not change when $x$ moves up to or down below $z_1$. If $\Label_D(z_1) < \Label_D(x)$, then since $D$ is southwest, there exists a cell $w$ in column $c$ of $T$ weakly above $z_1$ with $\Label_D(w) = \Label_D(z_1)$, contradicting that $x$ will have greater label. Similarly, if $\Label_D(z_1) = \Label_D(x)$, then since $\Label_D(y_1) < \Label_D(z_1) = \Label_D(x)$, we arrive at the same contradiction, now with $\Label_D(w) = \Label_D(y_1)$. Thus the claim follows. Notice as well, since $x$ moves down, if $\Label_D(T)$ is flagged, then so, too, is $\Label_D(\Ke_r(T))$.

  Now consider column $c-1$. If there is no cell in column $c-1$, row $r$, then $\Label_D$ agrees on both $T$ and $\Ke_r(T)$ in column $c-1$, and so too for all columns to the left, proving the theorem for this case. Otherwise, let $n \geq 1$ be maximal such that there exist $y_1,\ldots,y_n$ in row $r$ with $y_i$ in column $c-i$. Since $\Ke_r$ acts on $x$, each cell $y_i$ must be $r$-paired with some cell $z_i$ in row $r+1$ weakly right of $y_i$ and, since $x$ is not $r$-paired, strictly left of $x$. Thus there exist $z_1,\ldots,z_n$ in row $r+1$ with $z_i$ in column $c-i$; see Fig.~\ref{fig:to-left}. We must have $\Label_D(z_i) \leq \Label_D(z_{i-1})$ for $i>1$, since by the southwest condition there is a cell in column $c-i$ with label $\Label_D(z_{i-1})$ and the labeling algorithm would prefer $z_i$ unless it is already labeled, and by the same logic $\Label_D(y_i) \leq \Label_D(y_{i-1})$. If $\Label_D(y_1) < \Label_D(x)$, then moving $x$ up or down does not change the labeling algorithm in column $c-1$, and so there is no change further to the left either, proving the theorem for this case as well. Thus we may assume $\Label_D(y_1) \geq \Label_D(x)$, and we split into two dual cases.
  
  \begin{figure}[ht]
    \begin{center}
      \begin{tikzpicture}[xscale=0.45,yscale=0.45]
        \node at (-1,2) (c1) {$\scriptstyle c$};
        \node at (-8.5,1) (r2) {$\scriptstyle r+1$};
        \node at (-8.5,0) (r1) {$\scriptstyle r$};
        \node at (-1,1) (x)  {$\bigcir{x}$};
        \node at (-1,0) (x1)  {$-$};
        \node at (-2,1) (z1)  {$\bigcir{z_1}$};
        \node at (-4,1) (zm)  {$\bigcir{z_m}$};
        \node at (-6,1) (zn)  {$\bigcir{z_n}$};
        \node at (-7,0) (y)  {$-$};
        \node at (-2,0) (y1)  {$\bigcir{y_1}$};
        \node at (-4,0) (ym)  {$\bigcir{y_m}$};
        \node at (-6,0) (yn)  {$\bigcir{y_n}$};
        \draw[thick,dotted] (z1) -- (zm);  
        \draw[thick,dotted] (zm) -- (zn);  
        \draw[thick,dotted] (y1) -- (ym);  
        \draw[thick,dotted] (ym) -- (yn);  
      \end{tikzpicture}
      \caption{\label{fig:to-left}Situation in $T$ when $\Ke_r$ acts on $x$, where $\Label_D(z_1) = \cdots = \Label_D(z_{m})$ and $\Label_D(z_{m}) < \Label_D(y_{i})$ for $i<m$.}
    \end{center}
  \end{figure}

  Suppose $\Label_D$ is well-defined on $T$ and regard $x$ as lying in row $r+1$. In this case $\Label_D(y_1) > \Label_D(x)$ as it cannot be equal since $x$ lies above $y_1$. Thus there exists $n \geq m \geq 1$ maximal such that $\Label_D(y_i) > \Label_D(x)$. Since $\Label_D(y_i) > \Label_D(x) \geq \Label_D(z_i)$, we must in fact have $\Label_D(z_i) = \Label_D(z_{i-1})$ for $i>1$ and $\Label_D(z_1) = \Label_D(x)$ by Definition~\ref{def:kohnert-tableaux}\ref{i:invert}, which applies by Theorem~\ref{thm:stable}. Thus in columns $c-i$ of $\Ke_r(T)$ for $i=1,\ldots,m$, $\Label_D$ will swap the labels on $y_i,z_i$, with no other changes within the columns, so that from column $c-m-1$ to the left the labeling remains the same. Hence $\Label_D(x)$ is well-defined on $\Ke_r(T)$. Moreover, since upward moving labels replace smaller labels, if $\Label_D$ is flagged on $T$ then $\Label$ is flagged on $\Ke_r(T)$.

  Suppose $\Label_D$ is well-defined on $\Ke_r(T)$ and regard $x$ as lying in row $r$. Then $\Label_D(y_1) \leq \Label_D(x)$ as they are adjacent in the same row, and since the reverse inequality also holds, we in fact have $\Label_D(y_1) = \Label_D(x)$. Thus there exists $n \geq m \geq 1$ maximal such that $\Label_D(y_i) = \Label_D(x)$. Since there is no cell in row $r+1$, column $c$, we must have $\Label_D(y_i) < \Label_D(z_i)$, since otherwise we contradict Definition~\ref{def:kohnert-tableaux}\ref{i:invert}, which applies by Theorem~\ref{thm:stable}. Thus in columns $c-i$ of $T$ for $i=1,\ldots,m$, $\Label_D$ will swap the labels on $y_i,z_i$, with no other changes within the columns, so that from column $c-m-1$ to the left the labeling remains the same. Hence $\Label_D(x)$ is well-defined on $T$, and the theorem follows.
\end{proof}

We can now tighten Theorem~\ref{thm:embed} by identifying a Demazure crystal onto which each connected component of the Kohnert crystal maps.

\begin{theorem}
Let $D$ be a southwest diagram and $\Koh\subseteq \KD(D)$ any connected component of the Kohnert crystal on $\KD(D)$. Then rectification is a well-defined, weight-preserving bijection
  \[ \rect:\Koh \stackrel{\sim}{\longrightarrow} \B_w(\lambda) \]
  satisfying $\rect(\Ke_i(T)) = \e_i(\rect(T))$ for all $T\in\Koh$ and all $i \geq 1$, where $w\cdot\lambda = \wt(\D(\rect(\Label_D(U))))$ for $U$ the unique highest weight diagram on $\Koh$. In particular, the Kohnert crystal on $\KD(D)$ is a disjoint union of Demazure crystals.
  \label{thm:main}
\end{theorem}

\begin{proof}
    By Theorem~\ref{thm:label-necessary}, $\Label_D$ is well-defined and flagged on each $T\in\Koh\subseteq\KD(D)$.  By Theorem~\ref{thm:stable}, $\rect(\Label_D(T))$ is a Kohnert diagram for each $T\in\Koh$, and by Lemma~\ref{lem:label-crystal}, $\D(\rect(\Label_D(T))) = \D(\rect(\Label_D(U)))$ for each $T\in\Koh$. Thus every connected component $\Koh$ of the Kohnert crystal $\KD(D)$ embeds under rectification into $\B_w(\lambda)$, where $w\cdot\lambda = \wt(\D(\rect(\Label_D(U))))$ for $U$ the unique highest weight diagram on $\Koh$.

  Suppose $T\in\B_w(\lambda)$. Then there exists some sequence of row indices $r_1,\ldots,r_m$ such that $\Ke_{r_m} \cdots \Ke_{r_1} (T) = \rect(U)$ is the highest weight diagram for $\B_w(\lambda)$. By Lemma~\ref{lem:KD-pairing}, we may construct a diagram $S$ such that $\Ke_{r_m} \cdots \Ke_{r_1} (S) = U$, and by Theorem~\ref{thm:commute}, $\rect(S)=T$. By Lemma~\ref{lem:label-crystal}, $\Label_D$ is well-defined on $S$ since, by Theorem~\ref{thm:label-necessary} $\Label_D$ is well-defined on $U$. By Theorem~\ref{thm:stable}, $\rect(\Label_D(S))$ is the proper labeling of $T$ with respect to $\D(\rect(\Label_D(U)))$. Since $T\in\B_w(\lambda)$, this is flagged, and so by Lemma~\ref{lem:flagged}, $\Label_D(S)$ is flagged. Therefore, by Theorem~\ref{thm:label-sufficient}, $S \in \KD(D)$, and so $\B_w(\lambda) = \Koh$.
\end{proof}

Taking characters, we have now proved Conjecture~\ref{conj:demazure}.

\begin{corollary}
  For $D$ a southwest diagram, the Kohnert polynomial $\kohnert_{D}$ is nonnegative sum of Demazure characters.
  \label{cor:main}
\end{corollary}

Assaf and Searles \cite{AS19} suggest that the Kohnert polynomial of a southwest diagram $D$ is, in fact, the character of the Schur module for $D$ as defined by Reiner and Shimozono \cite{RS98}. This Demazure positivity result for southwest Kohnert polynomials lends further support to that conjecture.

\section{Explicit formulas}
%
\label{sec:applications}

We use labelings to give explicit formulas for Demazure expansions of southwest Kohnert polynomials with application motivated by Schubert polynomials.

\subsection{Yamanouchi diagrams}
\label{sec:yamanouchi}

By Theorem~\ref{thm:hwt-Kohnert}, each connected component of a southwest Kohnert crystal has a unique highest weight diagram. However, unlike in the tableau crystal setting, the highest weight does not determine the isomorphism class of the \emph{Demazure} crystal. For this, we require another concept.

\begin{definition}
  Given a Demazure crystal $\B_w(\lambda)$, an element $b \in \B_w(\lambda)$ is a \newword{Demazure lowest weight element} if $\wt(b) = w \cdot \lambda$.  
  \label{def:Dlwt}
\end{definition}

Each connected Demazure crystal $\B_w(\lambda)$ has a unique Demazure lowest weight element. Moreover, that element $b$ determines both the partition $\lambda = \mathrm{sort}(b)$ and the permutation $w$ which is the shortest (in Coxeter length) such that $w \cdot \lambda = \wt(b)$. Thus these elements are the correct analog of highest weight elements for tableaux crystals. However, while a Demazure lowest weight element $b$ necessarily satisfies $f_i(b)=0$ for all $i$, but this is not always sufficient to characterize $b$ since a Demazure crystal can have multiple lowest weights, as seen in Fig.~\ref{fig:demazure-crystal} with $\B_{312}(3,2,0)$. Thus we introduce a new concept specific to the Kohnert crystal to allow us to identity Demazure lowest weight elements in $\KD(D)$.


\begin{definition}
  A diagram $Y\in\KD(D)$ is \newword{Yamanouchi with respect to $D$} if $\rect(\Label_D(Y))$ is a super-standard composition diagram.
  \label{def:yamanouchi}
\end{definition}

By Theorem~\ref{thm:label-necessary}, $\Label_D$ is well-defined (and flagged) on $Y$, making this concept well-defined. For example, Fig.~\ref{fig:yamanouchi} shows the rectification of a Yamanouchi diagram, confirming correct use of the term, where steps have been consolidated to save space. 

\begin{figure}[ht]
  \begin{displaymath}
    \begin{array}{llllll}
      \cirtab{%
        & & 5 & 5 & 5\\
        4 & & 4 & 4 & 4 & & \leftball{red}{7} & \cball{red}{7} \\
        & \\
        2 & 2 & 2 & 7 & & 4 & \\
        & \\\hline } &
      \cirtab{%
        & & 5 & 5 & \cball{green}{5} \\
        4 & & 4 & 4 & 4 & \cball{green}{7} & \cball{green}{7} \\
        & \\
        2 & 2 & 2 & 7 & & 4 & \\
        & \\\hline } &
      \cirtab{%
        & & \leftball{red}{5} & \cball{red}{5} & \cball{red}{5} \\
        4 & & \leftball{red}{4} & \cball{red}{4} & \cball{red}{4} & \cball{red}{5} & \cball{red}{5} \\
        & \\
        2 & 2 & 2 & 7 & & \leftball{red}{4} & \\
        & \\\hline } &
      \cirtab{%
        & \leftball{red}{5} & \cball{red}{5} & \cball{red}{5} \\
        4 & 4 & 4 & 4 & 5 & 5 \\
        & \\
        2 & 2 & \cball{green}{2} & \cball{green}{7} & 4 \\
        & \\\hline }  &
      \cirtab{%
        5 & 5 & 5\\
        4 & 4 & 4 & \cball{green}{4} & \cball{green}{5} & \cball{green}{5} \\
        & \\
        2 & 2 & 2 & \cball{green}{2} & \cball{green}{4} \\
        & \\\hline } &
      \cirtab{%
        5 & 5 & 5\\
        4 & 4 & 4 & 4 & 4 & 4 \\
        & \\
        2 & 2 & 2 & 2 & 2 \\
        & \\\hline } 
    \end{array}
  \end{displaymath}
  \caption{\label{fig:yamanouchi}A Yamanouchi diagram (left) and its rectification (right).}
\end{figure}

\begin{theorem}
  For $D$ a southwest diagram and $Y\in\KD(D)$, $Y$ is Yamanouchi if and only if $Y$ is a Demazure lowest weight element. In particular, for $D$ southwest we have
  \begin{equation}
    \kohnert_D = \sum_{Y \in \Yam(D)} \key_{\wt(Y)} ,
    \label{e:kohnert-yam}
  \end{equation}
  where $\Yam(D) \subset \KD(D)$ is the set of Yamanouchi diagrams for $D$.
  \label{thm:yamanouchi}
\end{theorem}

\begin{proof}
  By Theorem~\ref{thm:embed}, the result follows if there exists a unique Yamanouchi diagram $Y$ on each connected component of the Demazure crystal. If a Yamanouchi $Y$ exists, then it is unique by Theorems~\ref{thm:commute} and~\ref{thm:stable} using the injectivity of rectification on crystal components. By Theorem~\ref{thm:hwt-Kohnert}, we may set $U$ to be the unique highest weight element on the component. By Theorem~\ref{thm:label-necessary}, $\Label_D$ is well-defined and flagged on $U$ since $U\in\KD(D)$ by Theorem~\ref{thm:hwt-Kohnert}. By Lemma~\ref{lem:label-crystal}, the proper labeling $\Label_D$ on $U$ becomes a proper labeling $\Label_{\comp{a}}$ on $\rect(U)$ for some weak composition $\comp{a}$. By Theorem~\ref{thm:crystal-dem}, $\rect(U) = \Ke_{s_n} \cdots \Ke_{s_1}(\D(\comp{a}))$ for some sequence of row indices. By Lemma~\ref{lem:KD-pairing} there exists a unique diagram $Y$ such that $\Ke_{i_m} \cdots \Ke_{i_1}(Y) = U$ and, by Theorem~\ref{thm:commute}, we have $\rect(Y) = \D(\comp{a})$. Since $\Label_D$ is well-defined on $U$, by Lemma~\ref{lem:label-crystal} $\Label_D$ is well-defined on $Y$ and rectifies to the super-standard labeling of $\D(\comp{a})$ and as such is flagged. Thus by Theorem~\ref{thm:stable}, $\Label_D(Y)$ is also flagged, and so by Theorem~\ref{thm:label-sufficient} $Y\in\KD(D)$ and so is Yamanouchi. 
\end{proof}

\begin{figure}[ht]
  \begin{displaymath}
    \arraycolsep=2\cellsize
    \begin{array}{ccccc}
      \cirtab{ & & 6 & & & 6 \\ & & 5 \\ \\ 3 & & 3 & 3 &     \\ 2 &         \\ & \\\hline} &
      \cirtab{ & & 6 & & & 6 \\ & &   \\ \\ 3 & & 3 & 3 &     \\ 2 & & 5     \\ & \\\hline} &
      \cirtab{ & & 6 & & &   \\ & &   \\ \\ 3 & & 5 &   & & 6 \\ 2 & & 3 & 3 \\ & \\\hline} &
      \cirtab{ & & 6 & & &   \\ & & 5 \\ \\ 3 & & 3 & 3 & & 6 \\ 2 &         \\ & \\\hline} &
      \cirtab{ & & 6 & & &   \\ & &   \\ \\ 3 & & 3 & 3 & & 6 \\ 2 & & 5     \\ & \\\hline} 
    \end{array}
  \end{displaymath}
   \caption{\label{fig:Dlwt}The Yamanouchi diagrams for the leftmost diagram above.}
\end{figure}

For example, taking $D$ to be the leftmost diagram in Fig.~\ref{fig:Dlwt}, which coincides with the Rothe diagram for the permutation $13625847$, the Yamanouchi diagrams are those shown in the figure. Taking weights, we obtain the expansion
\[ \kohnert_D = \key_{(0,1,3,0,1,2)} + \key_{(0,2,3,0,0,2)} + \key_{(0,3,3,0,0,1)} + \key_{(0,1,4,0,1,1)} + \key_{(0,2,4,0,0,1)} . \]

Looking back at Theorem~\ref{thm:LS}, to compute the Demazure expansion of a Schubert polynomial from the increasing reduced word paradigm as in \eqref{e:nilkey}, one must generate all increasing reduced words, which, in practice, requires computing the entire set of reduced words for $w$. In contrast, to compute the Demazure expansion \eqref{e:kohnert-yam} from the Yamanouchi diagram paradigm, we simply apply Kohnert moves to $\D(w)$ without creating new rows and filter the results based on the Yamanouchi condition.

For example, we computed $\schubert_{13625847}$ using words in Fig.~\ref{fig:yam} which were found by searching all reduced words for $13625847$ for the increasing ones, then lifting the result. The Yamanouchi diagrams in Fig.~\ref{fig:Dlwt} were computed more concisely using Kohnert moves on the Rothe diagram for $13625847$, circumventing the computational overhead of computing all reduced words. Even for the permutation $13625847$, the computational savings is significant and makes the diagram paradigm for Demazure expansions far more tractable.

\subsection{Quasi-Yamanouchi diagrams}
\label{sec:qyam}

In addition to the Demazure expansion, another interesting basis into which southwest Kohnert polynomials expand nonnegatively is the fundamental slide basis. Assaf and Searles \cite[Definition~3.6]{AS17} introduced fundamental slide polynomials as a generalization of the fundamental quasisymmetric functions of Gessel \cite{Ges84} that form a basis for the full polynomial ring. 

\begin{definition}[\cite{AS17}]
  The \newword{fundamental slide polynomial} $\fund_{\comp{a}}$ is
  \begin{equation}
    \fund_{\comp{a}} = \sum_{\substack{b_1 + \cdots + b_k \geq a_1 + \cdots + a_k \ \forall k \\ \mathrm{flat}(b) \ \mathrm{refines} \ \mathrm{flat}(a)}} x_1^{b_1} \cdots x_n^{b_n},
    \label{e:fund-shift}
  \end{equation}
  where $\mathrm{flat}(\comp{a})$ denotes the composition obtained by removing all zero parts.
  \label{def:fund-shift}
\end{definition}

Assaf and Searles \cite[Definition~4.8]{AS19} describe the subset of Kohnert diagrams giving rise to the fundamental slide expansion of the Kohnert polynomial. 

\begin{definition}[\cite{AS19}]
  A diagram $T \in \KD(D)$ is \newword{quasi-Yamanouchi} if for every row $r$, either some cell in row $r+1$ lies weakly right of some cell in row $r$ or raising all cells in row $r$ up to row $r+1$ is not a Kohnert diagram for $D$.
  \label{def:FKD}
\end{definition}

Assaf and Searles \cite[Definition~4.11]{AS19} characterize diagrams for which the Kohnert polynomial expands nonnegatively into fundamental slide polynomials. Southwest diagrams easily satisfy the condition, and so \cite[Theorem~4.14]{AS19} gives the following.

\begin{theorem}[\cite{AS19}]
  For $D$ southwest, we have
  \begin{equation}
    \kohnert_D = \sum_{T \in \QYKD(D)} \fund_{\wt(T)},
    \label{e:FKD}
  \end{equation}
  where $\QYKD(D) \subset \KD(D)$ is the set of quasi-Yamanouchi Kohnert diagrams.
  \label{thm:FKD}
\end{theorem}

Using Theorems~\ref{thm:label-necessary} and \ref{thm:label-sufficient}, we can give a more direct characterization of quasi-Yamanouchi diagrams, and so a more direct formula for the fundamental slide expansion of Kohnert polynomials.

\begin{proposition}
  Given a diagram $T$ and a southwest diagram $D$, we have $T\in\QYKD(D)$ if and only if $\Label_D$ is well-defined and flagged on $T$ and for every row $r$ in which no cell in row $r+1$ lies weakly right of the leftmost cell $x$ in row $r$, we have $\Label_D(x)=r$. 
  \label{prop:qyam}
\end{proposition}

\begin{proof}
  If $T\in\QYKD(D)\subset\KD(D)$, then $\Label_D$ is well-defined and flagged for $T$ by Theorem~\ref{thm:label-necessary}. If no cell in row $r+1$ lies weakly right of the leftmost cell $x$ in row $r$, then $\Kf_r^*(T)$ raises all cells in row $r$ up to row $r+1$, and since $T\in\QYKD(D)$, we must have $\Kf_r^*(T)=0$, in which case $\Kf_r(T)=0$ as well. Tracking the label of $x$, we have $\Label_D(x) \geq r$ in $T$, but we must have $\rect(\Label_D)(x) < r+1$ in $\rect(T)$, and so $\Label_D(x)=r$ as claimed.

  Conversely, if $\Label_D$ is well-defined and flagged on $T$, then by Theorem~\ref{thm:label-sufficient} $T\in\KD(D)$. Let $x$ be the leftmost cell in row $r$, and suppose there is no cell weakly right of this in row $r+1$. By the conditions on $T$, we must have $\Label_D(x) = r$. Since $\Label_D$ is flagged, any cell above row $r$ must have label greater than $r$, and so $x$ will never label pair with a cell above it during rectification. Thus $x$ lies in column $1$ of $\rect(T)$ and has label $r$. Moreover, by the southwest condition on $D$, any cell in row $r+1$ left of $x$ in $T$ forces another cell with label $r$ in the same column, which by the labeling algorithm must lie weakly above row $r$, contradicting either the flagged condition or the choice of $x$ as the leftmost cell of row $r$. In particular, row $r+1$ is empty, both in $T$ and in $\rect(T)$. Since $\rect(\Label_D)(x)=r$, $\Kf_r^*(\rect(T))$ is not flagged, and so $\Kf_r^*(T)$ is not flagged by Theorem~\ref{thm:stable}. Thus by Theorem~\ref{thm:label-necessary}, $\Kf_r^*(T) \not\in\KD(D)$, showing $T\in\QYKD(D)$.
\end{proof}

\begin{corollary}
  For $D$ southwest, we have $D \in \Yam(D) \subset \QYKD(D) \subset \KD(D)$, reflecting the nonnegative expansions
  \begin{equation}
    \kohnert_D = \sum_{T \in \Yam(D)} \key_{\wt(T)} = \sum_{T \in \QYKD(D)} \fund_{\wt(T)} = \sum_{T\in\KD(D)} x_1^{\wt(T)_1} \cdots x_n^{\wt(T)_n} .
  \end{equation}
  \label{cor:qyam}
\end{corollary}

\subsection{Vexillary diagrams}
\label{sec:vexillary}

Lascoux and Sch{\"u}tzenberger define a class of permutations they call \emph{vexillary}, and Macdonald \cite[(1.27)]{Mac91} gives various equivalent characterizations of this concept, some of which make use of the Rothe diagram of a permutation.

\begin{proposition}[\cite{Mac91}]
  The following are equivalent for a permutation $w$:
  \begin{enumerate}
  \item the set of rows of $\D(w)$ is totally ordered by inclusion;
  \item the set of columns of $\D(w)$ is totally ordered by inclusion;
  \item there do not exist $1 \leq a < b < c < d$ such that $w_b < w_a < w_d < w_c$.
  \end{enumerate}
  When any of these holds for $w$, we say $w$ is a \newword{vexillary} permutation.
  \label{prop:vexillary}
\end{proposition}

Lascoux and Sch{\"u}tzenberger \cite{LS90} show the Schubert polynomial of a vexillary permutation is a Demazure character. To state the result precisely, the \newword{Lehmer code} of a permutation $w$, denoted by $\Le(w)$, is the weak composition whose $i$th part is the number of indices $j>i$ for which $w_i > w_j$.

\begin{theorem}[\cite{LS90}]
  The Schubert polynomial $\schubert_w$ is equal to a single Demazure character if and only if $w$ is vexillary, and in this case we have $\schubert_w = \key_{\Le(w)}$.
  \label{thm:LS-vexillary}
\end{theorem}

In light of Corollary~\ref{cor:main}, it is natural to ask when a southwest Kohnert polynomial is equal to a single Demazure character, for which we introduce the following extension of the term \emph{vexillary}.

\begin{definition}
  A diagram $D$ is \newword{vexillary} if the set of rows of $D$ is totally ordered by inclusion. 
  \label{def:vexillary}
\end{definition}

By Proposition~\ref{prop:vexillary}, Rothe diagrams of vexillary permutations are vexillary. Thus Theorem~\ref{thm:LS-vexillary} is a special case of the following.

\begin{theorem}
  Given a southwest diagram $D$, the Kohnert polynomial $\kohnert_D$ is a single Demazure character if and only if $D$ is vexillary.
  \label{thm:vexillary}
\end{theorem}

\begin{proof}
  For $D$ southwest, the labeling $\Label_D$ on $D$ is invariant under relabeling for every column since the rectified labeling $\rect(\Label_D)$ of $\rect(D)$ is $\Label_{\rect(D)}$. To see this, let $x$ be a cell in column $c+1$, say in row $r$. If there is a cell $y$ in column $c$, row $r$, then since $\Label(x) = r = \Label(y)$, $x$ and $y$ will be label $c$-paired. Otherwise, by the southwest condition, there is no cell weakly above $x$ in column $c$, and so $x$ is not label $c$-paired. Thus no labels change.

  Suppose $D$ is vexillary, and let $T\in\KD(D)$. By Theorem~\ref{thm:label-necessary}, $\Label_D$ is well-defined on $T$. We claim it is also invariant under relabeling for every column, and so $\Rect^*_c(\Label_D) = \Label_{\Rect^*_c(D)}$ on $\Rect^*_c(T)$. In particular, by Theorem~\ref{thm:stable}, this means $\rect(T)\in\KD(\rect(D))$, and so by Theorem~\ref{thm:yamanouchi}, the Demazure expansion of the character has a single term. To prove the claim, suppose $x$ in column $c+1$ is label $c$-paired with $y$ in column $c$, where $\Label_D(y) < \Label_D(x)$ so that $x$ will lose its label. If there is another cell $z$ in column $c+1$ with $\Label_D(z) = \Label_D(y)$, then $z$ must lie weakly below $y$ by the labeling algorithm and so strictly below $x$ by the label pairing rule. Thus $z$ (or some other cell) will inherit $\Label_D(x)$, preserving the labelings. Otherwise, since $D$ is vexillary, there must be a cell $w$ in column $c$ with $\Label_D(w) = \Label_D(x)$. By the labeling algorithm, $w$ must lie weakly above $x$ and so by the pairing rule it must already be label $c$-paired with another cell, say $z$, in column $c+1$. Thus $z$ will take on $\Label_D(x)$ in column $c+1$, again showing labels are preserved.

  Now suppose $D$ is not vexillary. Choose $s$ then $r<s$ maximal such that these two rows are not ordered by containment. Let $c$ be the leftmost column with a cell in row $s$ but not in row $r$. Since $D$ is southwest, there is no cell in row $r$ weakly to the right of column $c$. By choice of $r,s$, every row in between is ordered by containment with both. Thus we may construct a diagram $T\in\KD(D)$ by pushing all cells in row $s$ weakly to the right of column $c$ down to row $r$, jumping over rows in between as needed. By Theorem~\ref{thm:label-necessary}, $\Label_D$ is well-defined on $T$. We claim $\Label_D(T)$ is not invariant under re-labeling, and so $\rect(T) \not\in\KD(\rect(D))$ by Theorem~\ref{thm:stable}. Thus there exists a Yamanouchi diagram $Y$ on the connected component of the Kohnert crystal for $\KD(D)$ containing $T$, and since $Y\neq D$ and $D$ is also Yamanouchi, by Theorem~\ref{thm:yamanouchi}, the Demazure expansion of the Kohnert polynomial has more than one term. To prove the claim, since $D$ is not vexillary, we may let $c_0$ be the rightmost column of $D$ with a cell in row $r$ but not in row $s$, and since $D$ is southwest, we have $c_0<c$. By choice of $c_1,c$, we must have the number of cells in rows $r,s$ of $D$ strictly between columns $c_1,c$. Thus $\rect(\Label_D(T_{>c_1})) = \Label_{\rect(D_{>c_1})}$ on $\rect(T_{>c_1})$. However, at column $c_1$, the cells in row $s$ labeled $s$ rectify left but the cells in row $r$ labeled $s$ do not, and so they are re-labeled. since $D$ is southwest, we have strictly more cells in row $r$ than in row $s$ weakly left of column $c_1$, and the labels are never restored. Thus $T$ is not invariant under re-labeling.
\end{proof}

Comparing the first two characterizations in Proposition~\ref{prop:vexillary}, if a southwest diagram is vexillary, then we can permute the columns, maintaining this condition, until the diagram becomes a composition diagram. Theorem~\ref{thm:vexillary} states doing so does not change the Kohnert polynomial, giving rise to the following.

\begin{lemma}
  Let $D$ be a southwest diagram, and let $c$ be a column such that columns $c,c+1$ are ordered by inclusion. Then $\mathfrak{s}_c \cdot D$, the diagram obtained from $D$ by permuting columns $c,c+1$, is southwest and $\kohnert_D = \kohnert_{\mathfrak{s}_c \cdot D}$. 
  \label{lem:vexillary}
\end{lemma}

\begin{proof}
  A violation of the southwest condition for columns $c_1,c_2$ forces those columns to be incomparable with respect to inclusion, and so $\mathfrak{s}_c \cdot D$ is southwest whenever $D$ is and columns $c,c+1$ are ordered by inclusion. Moreover, in this case the rectification operator $\Rect_c^*$ is the identity if the smaller column is $c$ and swaps the columns if the larger column is $c$, and this holds even as we apply Kohnert moves since, by Theorem~\ref{thm:label-necessary}, we may label any $T\in\KD(D)$ by $\Label_D(T)$, and cells in adjacent columns with the same label must have the left weakly above the right by the labeling algorithm. Thus we have equality for the multisets of diagrams $\{\!\{\rect(T) \mid T\in\KD(D)\}\!\} = \{\!\{\rect(T) \mid T\in\KD(\mathfrak{s}_c \cdot D)\}\!\}$, giving the result.
\end{proof}

It is a natural consider the equivalence relation $\sim$ on diagrams such that $D_1 \sim D_2$ if and only if $\kohnert_{D_1} = \kohnert_{D_2}$. In this context, Lemma~\ref{lem:vexillary} states $D \sim \mathfrak{s}_c \cdot D$ whenever columns $c,c+1$ are ordered by inclusion. It is an open question whether the transitive close these simple equivalences generate the relation.

%
%

\bibliographystyle{plain} 
\bibliography{kohnert_crystals}

\end{document}